\documentclass[a4paper,12pt,reqno]{amsart}
\usepackage[hyphens]{url}
\usepackage{enumerate}
\usepackage{booktabs}
\usepackage{dcolumn}
\usepackage{verbatim}
\usepackage{amsthm}

\def\ve#1{\mathchoice{\mbox{\boldmath$\displaystyle\bf#1$}}%
{\mbox{\boldmath$\textstyle\bf#1$}}%
{\mbox{\boldmath$\scriptstyle\bf#1$}}%
{\mbox{\boldmath$\scriptscriptstyle\bf#1$}}}

\newcommand\Z{\mathbb Z}
\newcommand\N{\mathbb N}
\newcommand\R{\mathbb R}
\newcommand\Q{\mathbb Q}
\newcommand{\Proj}{{\mathbb P}}
\newcommand{\lin}{\operatorname{lin}}
\DeclareMathOperator{\card}{Card}

\DeclareMathOperator{\conv}{conv}   
\DeclareMathOperator{\vol}{vol}     
\DeclareMathOperator{\Res}{Res}  

\let\epsilon=\varepsilon
\usepackage{ifthen}
\makeatletter
\newcommand{\DeclareBracket}[3]{
  \newcommand{#1}[2][]{%
  \ifthenelse%
  {\equal{##1}{}}%
  {\left#2##2\right#3}%
  {\csname ##1l\endcsname#2##2\csname ##1r\endcsname#3}}}
\makeatother
\DeclareBracket\abs||           
\DeclareBracket\norm\|\|        
\DeclareBracket\floor\lfloor\rfloor
\DeclareBracket\ceil\lceil\rceil
\DeclareBracket\set\{\}
\DeclareBracket\paren()
\DeclareBracket\inner\langle\rangle
\DeclareBracket\fractional\{\}

\newcommand\C{\mathbb C}
\newcommand\Order{\mathrm O}     

\newenvironment{inputlist}
  {\begin{enumerate}[\quad\rm({I}$_\bgroup 1\egroup$)]}
  {\end{enumerate}}
\newenvironment{outputlist}
  {\begin{enumerate}[\quad\rm({O}$_\bgroup 1\egroup$)]}
  {\end{enumerate}}

\newcommand\minavgmax[3]{$\begin{array}[t]{@{}D{.}{.}{1}@{}}#1\\\bf #2\\#3
  \end{array}$}

\newtheorem{theorem}{Theorem}%
\makeatletter
\newtheorem{lemma}{Lemma}
\renewcommand*{\c@lemma}{\c@theorem}
\renewcommand*{\p@lemma}{\p@theorem}

\renewcommand*{\c@conjecture}{\c@theorem}
\renewcommand*{\p@conjecture}{\p@theorem}

\newtheorem{proposition}{Proposition}
\renewcommand*{\c@proposition}{\c@theorem}
\renewcommand*{\p@proposition}{\p@theorem}

\newtheorem{corollary}{Corollary}
\renewcommand*{\c@corollary}{\c@theorem}
\renewcommand*{\p@corollary}{\p@theorem}

\renewcommand*{\c@observation}{\c@theorem}
\renewcommand*{\p@observation}{\p@theorem}

\theoremstyle{definition}

\renewcommand*{\c@problem}{\c@theorem}
\renewcommand*{\p@problem}{\p@theorem}

\renewcommand*{\c@definition}{\c@theorem}
\renewcommand*{\p@definition}{\p@theorem}

\newtheorem{remark}{Remark}
\renewcommand*{\c@remark}{\c@theorem}
\renewcommand*{\p@remark}{\p@theorem}

\renewcommand*{\c@example}{\c@theorem}
\renewcommand*{\p@example}{\p@theorem}

\renewcommand*{\c@algorithm}{\c@theorem}
\renewcommand*{\p@algorithm}{\p@theorem}

\makeatother

\usepackage[breaklinks=true,colorlinks,citecolor=blue]{hyperref}
\renewcommand{\d}{\,\mathrm{d}}

\title[How to Integrate a Polynomial over a Simplex]{How to Integrate\\ a Polynomial over a Simplex}
\author{V. Baldoni}
\address{Velleda Baldoni: Dipartimento di Matematica, Universit\`a degli studi di  Roma ``Tor Vergata'',
Via della ricerca scientifica 1, I-00133, Italy}
\email{baldoni@mat.uniroma2.it}
\author{N. Berline}
\address{Nicole Berline: Centre de Math\'ematiques Laurent Schwartz, \'Ecole Polytechnique, 91128 Palaiseau Cedex, France}
\email{nicole.berline@math.polytechnique.fr}
\author{J. A. De Loera}
\address{Jes\'us A. De Loera:  Department of
  Mathematics, University of California,
  Davis, One Shields Avenue, Davis, CA, 95616, USA}
\email{deloera@math.ucdavis.edu}
\author{M. K\"oppe}
\address{Matthias~K\"oppe:  Department of
  Mathematics, University of California,
  Davis, One Shields Avenue, Davis, CA, 95616, USA}
\email{mkoeppe@math.ucdavis.edu}
\author{M. Vergne}
\address{Mich\`ele Vergne: Centre de Math\'ematiques Laurent Schwartz, \'Ecole Polytechnique, 91128 Palaiseau Cedex, France}
\email{vergne@math.polytechnique.fr}

\begin{document}

\begin{abstract} This paper settles the computational complexity of the problem of
integrating a polynomial function $f$ over a rational simplex. We
prove that the problem is $\mathrm{NP}$-hard for arbitrary
polynomials via a generalization of a theorem of Motzkin and Straus.
On the other hand, if the polynomial depends only on a fixed number
of variables, while its degree and the dimension of the simplex are
allowed to vary, we prove that integration can be done in polynomial
time. As a consequence, for polynomials of fixed total degree, there
is a polynomial time algorithm as well. We conclude the article with
extensions to other polytopes and discussion of other available
methods.
\end{abstract}

\maketitle

\section{Introduction}

Let $\Delta$ be a $d$-dimensional rational simplex inside $\R^n$ and
let $f\in \Q[x_1,\dots,x_n]$ be a polynomial with rational
coefficients. We consider the problem of how to efficiently compute
the \emph{exact} value of the integral of the polynomial $f$ over
$\Delta$, which we denote by $\int_{\Delta} f \d m$. We use here the
\emph{integral Lebesgue measure} $\d m$ on the affine hull $\inner{
\Delta}$
 of the simplex $\Delta$, defined below in
section \ref{lebesgue}. This normalization of the measure occurs
naturally in Euler--Maclaurin formulas for a polytope $P$, which
relate sums over the lattice points of~$P$ with certain integrals
over the various faces of~$P$. For this measure, the volume of the
simplex and every integral of a polynomial function with rational
coefficients are  \emph{rational numbers}. Thus the result has a
representation in the usual (Turing) model of computation. This is
in contrast to other normalizations, such as the induced Euclidean
measure, where irrational numbers appear.

The main goals of this article are to discuss the computational
complexity of the problem and to provide methods to do the
computation that are both theoretically efficient and have
reasonable performance in concrete examples.

Computation of integrals of polynomials over polytopes is
fundamental for many applications.   We already mentioned summation
over lattice points of a polytope.  They also make an appearance in
recent results in optimization problems connected to moment matrices
\cite{laurentsurvey}. These integrals are also commonly computed in
finite element methods, where the domain  is decomposed into cells
(typically simplices) via a mesh and complicated functions are
approximated by polynomials (see for instance \cite{ZienkTay}). When
studying a random univariate polynomial $p(x)$ whose coefficients
are independent random variables in certain intervals, the
probability distribution for the number of real zeros of $p(x)$ is
given as an integral over a polytope \cite{bharuchaetal}. Integrals
over polytopes also play a very important role in statistics, see,
for instance, \cite{lin-sturmfels-xu:marginal-likelihood}. Remark
that among all polytopes, simplices are the fundamental case to
consider for integration since any convex polytope can be
triangulated into finitely many simplices.

Regarding the computational complexity of our problem, one can ask
what happens with integration over arbitrary polytopes.  It is very
educational to look first at the case when $f$ is the constant
polynomial~$1$, and the answer is simply a volume.     It has been
proved that already computing the volume of polytopes of varying
dimension is $\#\mathrm{P}$-hard
\cite{dyerfrieze88,brightwellwinkler91,khachiyan93,lawrence91}, and
that even approximating the volume is hard~\cite{elekes86}. More
recently in \cite{rademacher} it was proved that computing the
centroid of a polytope is $\#\mathrm{P}$-hard. In contrast, for a
simplex, the volume is given by a determinant, which can  be
computed in polynomial time. One of the key contributions of this
paper is to settle the computational complexity of integrating a
non-constant polynomial over a simplex.  Before we can state our
results let us understand better the input and output of our
computations.  Our output will always be the rational number
$\int_\Delta f \d m$ in the usual binary encoding. The
$d$-dimensional input simplex will be represented by its vertices
$\ve s_1,\dots,\ve s_{d+1}$ (a $V$-representation) but note that, in
the case of a simplex, one can go from its representation as a
system of linear inequalities (an $H$-representation) to a
$V$-representation in polynomial time, simply by computing the
inverse of a matrix.

Thus the encoding size of $\Delta$ is given by the number of
vertices, the dimension, and the largest binary encoding size of the
coordinates among vertices. Computations with polynomials also
require that one specifies concrete data structures for reading the
input polynomial and to carry on the calculations. There are several
possible choices. One common representation of a polynomial is as a
sum of monomial terms with rational coefficients. Some authors
assume the representation is \emph{dense} (polynomials are given by
a list of the coefficients of all monomials up to a given total
degree~$r$), while other authors assume it is \emph{sparse}
(polynomials are specified by a list of exponent vectors of
monomials with non-zero coefficients, together with their
coefficients). Another popular representation is by
\emph{straight-line programs}.  A straight-line program which
encodes a polynomial is, roughly speaking, a program without
branches which enables us to evaluate it at any given point (see
\cite{burgisseretal,matera} and references therein). As we explain
in Section \ref{preliminar}, general straight-line programs are
\emph{too compact} for our purposes, so instead we restrict to a
subclass we call \emph{single-intermediate-use (division-free)
  straight-line programs} or \emph{SIU straight-line programs} for short. The
precise definition and explanation will appear in Section~\ref{preliminar}
 but for now the reader should think that polynomials
are represented as fully parenthesized arithmetic expressions
involving binary operators $+$ and $\times$.\smallbreak

Now we are ready to state our first result.

\begin{theorem}[Integrating general polynomials over a simplex is hard]
  \label{generalMhard}
  The following problem is NP-hard.
  Input:
  \begin{inputlist}
  \item numbers~$d, n\in\N$ in unary encoding,
  \item affinely independent rational vectors $\ve s_1,\dots,\ve s_{d+1} \in
    \Q^n$ in binary encoding,
  \item an SIU straight-line program~$\Phi$ encoding a
    polynomial~$f\in\Q[x_1,\dots,x_n]$ with rational coefficients.
  \end{inputlist}
  Output, in binary encoding:
  \begin{outputlist}
  \item the rational number $\int _\Delta f \d m$,
    where $\Delta\subseteq\R^n$ is the simplex with vertices~$\ve s_1,\dots,\ve s_{d+1}$
    and $\d m$ is the integral     Lebesgue measure of the rational affine subspace $\langle\Delta\rangle$.
  \end{outputlist}
\end{theorem}

But we can also prove the following positive results.

\begin{theorem}[Efficient integration of polynomials of fixed \emph{effective}
number of variables]
  \label{powersoflinear}
  For every fixed number~$D \in\N$, there exists a polynomial-time algorithm
  for the following problem.

\noindent  Input:
  \begin{inputlist}
  \item numbers~$ d,n,M \in\N$ in unary encoding,
  \item affinely independent rational vectors $\ve s_1,\dots,\ve s_{d+1} \in
    \Q^n$ in binary encoding,
  \item a polynomial~$f\in\Q[X_1,\dots,X_{D}]$ represented by either an SIU
    straight-line program~$\Phi$ of formal degree at most~$M$,
    or a sparse or dense monomial representation of total degree at most~$M$,
    \item a rational matrix $L$ with $D$ rows and $n$ columns in binary
      encoding, the rows of which define $D$~linear forms
    $\ve x\mapsto \langle \ell_j,\ve x\rangle$ on $\R^n$.
  \end{inputlist}
\noindent   Output, in binary encoding:
  \begin{outputlist}
  \item the rational number $\int _\Delta f(\langle \ell_1,\ve x\rangle,\dots, \langle \ell_D,\ve x\rangle) \d m $, where
    $\Delta\subseteq\R^n$ is the simplex with vertices~$\ve s_1,\dots,\ve
    s_{d+1}$     and $\d m$ is the integral     Lebesgue measure
    of the rational affine subspace $\langle\Delta\rangle$.
  \end{outputlist}
\end{theorem}
In particular, the computation of the integral of a power of
\emph{one} linear form can be done by a polynomial time algorithm.
This becomes  false already if  one considers  powers of a
{quadratic form} instead of powers of a {linear form}. Actually, we
prove Theorem \ref{generalMhard} by looking at powers $Q^M$ of the
Motzkin--Straus quadratic form of a graph.

Our method relies on properties of  integrals of  exponentials of
linear forms. A. Barvinok had previously investigated these
integrals and their computational complexity (see
\cite{Barvinok-1991}, \cite{Barvinok-1992}).

 As we will see later, when its degree is fixed,  a polynomial has a
 polynomial size representation in
either the SIU straight-line program encoding or the sparse or dense
monomial representation and one can switch between the three
representations efficiently.  The notion of formal degree of an SIU
straight-line program will be defined in Section~\ref{preliminar}.
\begin{corollary} [Efficient integration of polynomials of fixed degree]
  \label{fixMpolytime}
  For every fixed number~$M\in\N$, there exists a polynomial-time algorithm
  for the following problem.  Input:
  \begin{inputlist}
  \item numbers~$d,n\in\N$ in unary encoding,
  \item affinely independent rational vectors $\ve s_1,\dots,\ve s_{d+1} \in
    \Q^n$ in binary encoding,
  \item a polynomial~$f\in\Q[x_1,\dots,x_n]$ represented by either an SIU
    straight-line program~$\Phi$ of formal degree at most~$M$,
    or a sparse or dense monomial representation of total degree at
    most~$M$.
  \end{inputlist}
   Output, in binary encoding:
  \begin{outputlist}
  \item the rational number $\int _\Delta f(\ve x) \d m$, where
    $\Delta\subseteq\R^n$ is the simplex with vertices~$\ve s_1,\dots,\ve s_{d+1}$
    and $\d m$ is the integral     Lebesgue measure of the rational affine subspace $\langle\Delta\rangle$.
  \end{outputlist}
\end{corollary}
Actually, we  give two interesting methods  that prove
Corollary~\ref{fixMpolytime}. First, we simply observe that a
monomial with total degree $M$ involves at most $M$ variables. The
other  method  is related to the polynomial Waring problem: we
decompose a homogeneous polynomial of total degree $M$ into a sum of
$M$-th powers of linear forms.

In \cite{lasserre-avrachenkov2001} Lasserre and Avrachenkov compute
the integral  $\int _\Delta f(\ve x) \d m$ when $f$ is a homogeneous
polynomial, in terms of the corresponding polarized symmetric
multilinear form (Proposition \ref{th:lasserre-avrachenkov}). We
show  that their formula also leads to a proof of Corollary
\ref{fixMpolytime}. Furthermore, several other methods can be used
for  integration of polynomials of fixed degree. We discuss them in
Section \ref{otheralgsextensions}.

This paper is organized as follows: After some preparation in
Section \ref{preliminar}, the main theorems are proved in Section
\ref{mainthms}. In Section \ref{otheralgsextensions}, we discuss
extensions to other convex polytopes and give a survey of the
complexity of other algorithms. Finally, in
Section~\ref{s:computations}, we describe the implementation of the
two  methods  of Section \ref{mainthms}, and we report on a few
computational experiments.

\section{Preliminaries} \label{preliminar}

In this section we prepare for the proofs of the main results.

\subsection{Integral Lebesgue measure on a rational affine subspace of
$\R^n$}\label{lebesgue} On $\R^n$ itself  we consider the standard
Lebesgue measure, which gives volume $1$ to the fundamental domain
of the lattice $\Z^n$. Let $L$ be a rational linear subspace of
dimension $d\leq n$. We normalize the Lebesgue measure on~$L$, so
that the volume of the fundamental domain of the intersected lattice
$L\cap \Z^n$ is~$1$. Then for any affine subspace $L+ \ve a$
parallel to $L$, we define the \emph{integral Lebesgue measure}  $\d
m$  by translation. For example, the diagonal  of the unit square
has length $1$ instead of  $\sqrt{2}$.

\subsection{Encoding polynomials for integration}

We now explain our encoding of polynomials as SIU straight-line programs
and justify our use of this encoding. We say that a polynomial~$f$ is
represented as a (division-free) \emph{straight-line program}~$\Phi$ if there is a finite sequence
of polynomial functions of $\Q[x_1,\dots,x_n]$, namely $q_1,\dots,q_k$, the
so-called \emph{intermediate results}, such that each $q_i$ is either a variable $x_1,\dots,x_n$,  an element of~$\Q$,
or either the sum or the product of two preceding polynomials in the sequence
and such that $q_k = f$.
A straight-line program allows us to
describe in polynomial space polynomials which otherwise would need to be
described with exponentially many monomial terms. For example, think of
of the representation of $(x_1^2+\dots+x_n^2)^k$ as monomials versus its
description with only $3n+k+2$ intermediate results; see \autoref{tab:slp1}.
\begin{table}[t]
  \caption{The representation of $(x_1^2+\dots+x_n^2)^k$ as a straight-line
    program}
  \label{tab:slp1}
  \begin{minipage}{0.8\linewidth}
    \begin{displaymath}
      \begin{array}{r@{\;}c@{\;}lp{.5\linewidth}}
        \toprule
        \multicolumn{3}{c}{\text{Intermediate}}
        & \multicolumn{1}{l}{\text{Comment}} \\
        \midrule
        q_1 &=& 0 \\
        q_2 &=& x_1 \\
        q_3 &=& q_2 \cdot q_2 \\
        q_4 &=& q_1 + q_3
        & Thus $q_4 = x_1^2$. \\
        q_5 &=& x_2 \\
        q_6 &=& q_5 \cdot q_5 \\
        q_7 &=& q_4 + q_6
        & Now $q_7 = x_1^2 + x_2^2$.
        \\
        &\vdots\\
        q_{3n-1} &=& x_n \\
        q_{3n} &=& q_{3n-1} \cdot q_{3n-1} \\
        q_{3n+1} &=& q_{3n-2} + q_{3n}
        & Now $q_{3n+1} = x_1^2+\dots+x_n^2$.
        \\
        q_{3n+2} &=& 1 \\
        q_{3n+3} &=& q_{3n+2} \cdot q_{3n+1} \\
        q_{3n+4} &=& q_{3n+3} \cdot q_{3n+1} \\
        &\vdots \\
        q_{3n+k+2} &=& q_{3n+k+1} \cdot q_{3n+1}
        & Final result.\\
        \bottomrule
      \end{array}
    \end{displaymath}
  \end{minipage}
\end{table}
The number of intermediate results of a straight line program is
called its \emph{length}.  To keep track of constants we define the
\emph{size} of an intermediate result as one, unless the
intermediate result is a constant in which case its size is the
binary encoding size of the rational number. The \emph{size} of a
straight-line program is the sum of the sizes of the intermediate
results. The \emph{formal degree} of an intermediate result~$q_i$ is
defined recursively in the obvious way, namely as 0 if $q_i$ is a
constant of~$\Q$, as 1 if $q_i$ is a variable~$x_j$, as the maximum
of the formal degrees of the summands if $q_i$ is a sum, and the sum
of the formal degrees of the factors if~$q_i$ is a product.  The
\emph{formal degree of the straight-line program}~$\Phi$ is the
formal degree of the final result~$q_k$. Clearly the total degree of
a polynomial is bounded by the formal degree of any straight-line
program which represents it.

A favorite example to illustrate the benefits of a straight-line
program encoding is that of the symbolic determinant of an $n \times
n$ matrix. Its dense representation as monomials has size
$\Theta(n!)$ but it can be computed in $\Order(n^3)$ operations by
Gaussian elimination. See the book \cite{burgisseretal}  as a
reference for this concept.

From a monomial representation of a polynomial of degree $M$ and $n$
variables it is easy to encode it as a straight-line program: first,
by going in increasing degree we can write a straight-line program
that generates all monomials of degree at most $M$ in $n$ variables.
Then for each of them compute the product of the monomial with its
coefficient so the length doubles. Finally successively add each
term. This gives a final length bounded above by  four times the
number of monomials of degree at most $M$ in $n$ variables.

Straight-line programs are quite natural in the context of integration. One would certainly not expand
$(x_1^2+\dots+x_n^2)^k$ to carry on numeric integration when we can easily
evaluate it as a function.  More importantly, straight-line programs are
suitable as an input and output encoding and data structure in certain
symbolic algorithms for computations with polynomials, like factoring; see
\cite{burgisseretal}.  Since straight-line programs can be very compact,
the algorithms can handle polynomials whose input and output encodings
have an exponential size in a sparse monomial representation.

However, a problem with straight-line programs is that this input encoding can
be so compact that the output of many computational questions cannot be written
down efficiently in the usual \emph{binary encoding}.
 For example, while one can encode the polynomial $x^{2^k}$
with a straight-line program with only~$k+1$ intermediate results
(see \autoref{tab:slp1}),
\begin{table}[t]
  \caption{The representation of a straight-line program for $x^{2^k}$, using
    iterated squaring}
  \label{tab:slp1}
  \begin{minipage}{0.8\linewidth}
    \begin{displaymath}
      \begin{array}{r@{\;}c@{\;}lp{.5\linewidth}}
        \toprule
        \multicolumn{3}{c}{\text{Intermediate}}
        & \multicolumn{1}{l}{\text{Comment}} \\
        \midrule
        q_1 &=& x \\
        q_2 &=& q_1 \cdot q_1 \\
        q_3 &=& q_2 \cdot q_2 \\
        &\vdots\\
        q_{k+1} &=& q_k \cdot q_k
        & Final result.\\
        \bottomrule
      \end{array}
    \end{displaymath}
  \end{minipage}
\end{table}
when we compute the value of $x^{2^k}$ for
$x=2$, or the integral $\int_0^2 x^{2^k}\d x=2^{2^k+1}/(2^k+1)$, the
binary encoding of the output has a size of $\Theta(2^k)$. Thus the
output, given in binary, turns out to be exponentially bigger than
the input encoding. We remark that the same difficulty arises if we
choose a sparse input encoding of the polynomial where not only the
coefficients but also the exponent vectors are encoded in binary
(rather than the usual unary encoding for the exponent vectors).

This motivates the following variation of the notion of
straight-line program: We say a (division-free) straight-line
program is \emph{single-intermediate-use}, or SIU for short, if
every intermediate result is used only once in the definition of
other intermediate results.  (However, the variables $x_1,\dots,x_n$ can be
used arbitrarily often in the definition of intermediate results.)
With this definition, all ways to encode
the polynomial $x^{2^k}$ require at least $2^k$ multiplications.
An example SIU straight-line program is shown in \autoref{tab:siu-slp}.
\begin{table}[t]
  \caption{The representation of a single-intermediate-use straight-line
    program for $x^{2^k}$; note that the iterated squaring method cannot be
    used}
  \label{tab:siu-slp}
  \begin{minipage}{0.8\linewidth}
    \begin{displaymath}
      \begin{array}{r@{\;}c@{\;}lp{.5\linewidth}}
        \toprule
        \multicolumn{3}{c}{\text{Intermediate}}
        & \multicolumn{1}{l}{\text{Comment}} \\
        \midrule
        q_1 &=& x \\
        q_2 &=& x \\
        q_3 &=& q_1 \cdot q_2 & Now $q_3 = x^2$, and $q_1$ and $q_2$ cannot be
        used anymore. \\
        q_4 &=& x \\
        q_5 &=& q_3 \cdot q_4 & Thus $q_5 = x^3$.\\
        &\vdots\\
        q_{2^{k+1}-2} &=& x \\
        q_{2^{k+1}-1} &=& q_{2^{k+1}-3} \cdot q_{2^{k+1}-2}
          & Final result.\\
        \bottomrule
      \end{array}
    \end{displaymath}
  \end{minipage}
\end{table}
Clearly single-intermediate-use straight-line programs are
equivalent, in terms of expressiveness and encoding complexity, to
fully parenthesized arithmetic expressions using binary operators
$+$ and $\times$.

\subsection{Efficient computation of truncated product of an arbitrary number of polynomials in a fixed number of variables}
The following result will be used in several situations.

\begin{lemma} \label{truncated-product}
For every fixed number $D\in \N$, there exists a polynomial time
algorithm for the following problem.

\noindent Input: a number $M$  in unary encoding, a sequence of $k$
polynomials $P_j\in \Q[X_1,\dots, X_D]$ of total degree at most $M$,
in dense monomial representation.

\noindent Output: the product $P_1\cdots P_k$ truncated at degree
$M$.
\end{lemma}
\begin{proof}
We start with the product of the first two polynomials.  We compute
the monomials of degree at most $M$  in this product.
 This takes $\Order(M^{2D})$
elementary rational operations, and the maximum encoding length of
any coefficient in the product  is also polynomial in the input data
length. Then we multiply this truncated product with the next
polynomial, truncating at degree $M$,  and so on. The total
computation takes $\Order(k M^{2D})$ elementary rational operations.
\end{proof}

\section{Proofs of the main results}
\label{mainthms}

Our aim is to perform an efficient computation of $\int_\Delta f \d m$ where
$\Delta$ is a simplex and $f$ a polynomial. We will prove first that this is
not possible for $f$ of varying degree under the assumption that $\mathrm{P}
\not= \mathrm{NP}$.  More precisely, we prove that, under this assumption, an
efficient computation of $\int_\Delta Q^M \d m$ is not possible, where $Q$ is
a quadratic form and $M$ is allowed to vary.

In the next subsection we present an algorithm to efficiently
compute the integral $\int_\Delta f \d m$ in some particular
situations, most notably the case of arbitrary powers of linear
forms.

\subsection{Hardness for polynomials of non-fixed degree}

For the proof of Theorem \ref{generalMhard} we need to extend the
following well-known result of Motzkin and Straus
\cite{motzkin-straus}. In this section, we denote by $\Delta $ the
$(n-1)$-dimensional canonical simplex $\{\, \ve x \in \R^n : x_i
\geq 0, \ \sum_{i=1}^{n} x_i= 1\, \}$, and   we denote by $\d m$ the
Lebesgue measure on the hyperplane $\{\, \ve x \in \R^n :
\sum_{i=1}^{n} x_i= 1\, \}$, normalized so that  $\Delta$ has volume
$1$. For a function $f$ on $\Delta$, denote as usual
$\norm{f}_\infty=\max_{\ve x \in \Delta} \abs{f(\ve x)}$ and
$\norm{f}_p=(\int_\Delta \abs{f}^p \d m)^{1/p} $,  for $p\geq 1$.
Recall that the \emph{clique number} of a graph~$G$ is the largest number of vertices
of a complete subgraph of~$G$.

\begin{theorem}[Motzkin--Straus]
\label{motzkinstraus1} Let $G$ be a graph with $n$ vertices and
clique number $\omega(G)$. Let $Q_G(\ve x)$ be the Motzkin--Straus
quadratic form $\frac{1}{2} \sum_{(i,j) \in E(G)} x_ix_j$.   Then
$\norm{Q_G(\ve x)}_\infty =\frac{1}{2}(1-\frac{1}{\omega(G)})$.
\end{theorem}

Our first result might be of independent interest as it shows that
integrals of polynomials over simplices can carry very interesting
combinatorial information. This result builds on the theorem of
Motzkin and Straus, using the proof of the well-known relation
$\norm{f}_\infty = \lim_{p\to \infty}\norm{f}_p$.

\begin{lemma}\label{motzkinstraus}
\label{motzkinstraus2} Let $G$ be a graph with $n$ vertices and
clique number $\omega(G)$. Let $Q_G(\ve x)$ be the Motzkin--Straus
quadratic form.
  Then for $p \geq 4(e-1) n^3\ln(32 n^2)$, the clique number $\omega(G)$ is equal to  $\bigl\lceil
\frac{1} {1-2\norm{Q_G}_p} \bigr\rceil$.
\end{lemma}

To prove Lemma \ref{motzkinstraus2} we will first prove the
following intermediate result.

\begin{lemma} \label{helpfullemma}
For $\epsilon>0$ we have
$$
(\norm{Q_G}_\infty-\epsilon)(\frac{\epsilon}{4})^{(n-1)/p } \leq
\norm{Q_G}_p \leq \norm{Q_G}_\infty.
$$
\end{lemma}

\begin{proof} The right-hand side inequality follows from the normalization of the measure, as
$|Q(\ve x)| \leq \norm{Q_G}_\infty$, for all $\ve x\in
  \Delta$.

In order to obtain the other inequality,   we use  H\"older's
inequality
 $\int_{\Delta}
\abs{fg} \d m \leq \norm{f}_p \norm{g}_q$, where  $q$ is such that
$\frac{1}{p}+\frac{1}{q}=1$. For any (say) continuous function $f$
on $\Delta$, let us denote by $\Delta(f,\epsilon)$ the set $\{\,\ve
x \in \Delta: \abs{f(\ve x)}\geq \norm{f}_\infty -\epsilon\,\}$, and
take for  $g$ the characteristic function of $\Delta(f,\epsilon)$.
We obtain
\begin{equation} \label{fromholder}
(\norm{f}_{\infty}-\epsilon) (\vol{\Delta(f,\epsilon)})^{1/p} \leq
\norm{f}_p.
\end{equation}
Let $\ve a$ be a point of $\Delta$ where the maximum of $Q_G$ is
attained. Since $\frac{\partial Q_G}{\partial x_i}=\sum_{(i,j) \in
E(G)} x_j$ we know that $0 \leq \frac{\partial Q_G}{\partial x_i}
\leq 1$ for $\ve x\in \Delta$.  Since $\Delta$  is convex, we
conclude that for any $\ve x\in \Delta$,
$$ 0 \leq Q_G(\ve a) -Q_G(\ve x) \leq \sum_{i=1}^n \abs{a_i -x_i}. $$
Thus $\Delta(Q_G,\epsilon)$ contains the set $C_\epsilon=\{\, \ve
x\in\Delta : \sum_{i=1}^n \abs{a_i
  -x_i}<\epsilon \}$.
We claim that $\vol (C_\epsilon)\geq (\frac{\epsilon}{4})^{n-1}$ .
This claim proves the left inequality of the lemma when we apply it
to~\eqref{fromholder}.

Consider the dilated simplex $\frac{\epsilon/2}{1+\epsilon/2}\Delta$
and the  translated set $P_\epsilon=\frac{\ve
a}{1+\epsilon/2}+\frac{\epsilon/2}{1+\epsilon/2}\Delta$. Clearly
$P_\epsilon$ is contained in $\Delta$. Moreover, for $\ve x\in
P_\epsilon$, we have $\sum_{i=1}^n \abs{a_i
  -x_i}\leq \frac{\epsilon}{1+\epsilon/2}\leq \epsilon$, hence  $P_\epsilon$ is contained in $C_\epsilon$.
 Since  $\vol( \Delta)=1$ for the normalized measure,  the volume of
$P_\epsilon$ is equal to $(\frac{\epsilon/2}{1+\epsilon/2})^{n-1}$.
Hence $\vol (P_\epsilon ) \geq (\epsilon/4)^{n-1}$. This finishes
the proof.
\end{proof}

\begin{proof}[Proof of Lemma \ref{motzkinstraus2}]
 In the inequalities of
Lemma \ref{helpfullemma}, we substitute the relation
$\norm{Q_G}_\infty =\frac{1}{2}(1-\frac{1}{\omega (G)})$, given  by
Motzkin--Straus's
  theorem (Theorem~\ref{motzkinstraus1}). We obtain
$$ (\frac{1}{2}(1-\frac{1}{\omega (G)})-\epsilon)(\epsilon/4)^{\frac{n-1}{p} }\leq \norm{Q_G}_p \leq \frac{1}{2}(1-\frac{1}{\omega (G)}).$$
Let us rewrite these inequalities  as
\begin{equation}
\frac{1}{1-2\norm{Q_G}_p} \leq \omega (G) \leq
\frac{1}{1-\frac{2\norm{Q_G}_p} {(\epsilon/4)^{(n-1)/p}}-2
\epsilon}.
\end{equation}
We only need to prove that for $\epsilon = \frac{1}{8n^2}$ and  $p
\geq 4(e-1) n^3\ln(32 n^2)$ we have
\begin{equation} \label{mainclaim}
0 \leq L(p):= \frac{1}{1-\frac{2\norm{Q_G}_p}
{(\epsilon/4)^{\frac{n-1}{p}}}-2\epsilon} -
\frac{1}{1-2\norm{Q_G}_p} <1.
\end{equation}
Let us write
\[
\delta_p=\norm{Q_G}_p (\frac{1}{(\epsilon/4)^{\frac{n-1}{p}}}
-1)=\norm{Q_G}_p ((32 n^2)^{\frac{n-1}{p}}-1).
\]

Thus $L(p)$ in Equation \ref{mainclaim} becomes now
\begin{equation} \label{modclaim}
L(p)=\frac{1}{1-2\norm{Q_G}_p} \big(\frac{1}{1-2
\frac{\delta_p+\epsilon} {1-2 \norm{Q_G}_p}}-1\big).
\end{equation}

Since $\norm{Q_G}_p \leq \frac{1}{2} (1-\frac{1}{\omega(G)}) \leq
\frac{1}{2} $, we have  a bound for $\delta_p$
\[
0 \leq \delta_p \leq \frac{1}{2} ((32 n^2)^{\frac{n-1}{p}}-1) .
\]
Let $A=(\frac{4}{\epsilon})^{n-1}= (32 n^2)^{n-1}$. Since we assumed
$p\geq 4(e-1) n^3\ln(32 n^2)$, we have $0\leq \frac{\ln A}{p}< 1$,
hence $0\leq  A^{1/p}-1< (e-1) \frac{ \ln A}{p}$. We obtain
$$
0 \leq \delta_p \leq \frac{e-1}{2} \frac{(n-1)\log(32 n^2)}{p} \leq
\frac{1}{8n^2}.
$$
 Since $\omega (G) \leq n$, we have $1 - 2\norm{Q_G}_p \geq 1/n$.
Hence we have
\[
\frac{2(\delta_p+\epsilon)}{1-2\norm{Q_G}_p} \leq \frac{1}{2n} \leq
\frac{1}{2}.
\]
Finally for any number $0< \alpha <1/2$ we have $\frac{1}{1-\alpha}
< 1+ 2\alpha$, hence applying this fact to Equation \ref{modclaim}
with $\alpha=2 \frac{\delta_p+\epsilon} {1-2 \norm{Q_G}_p}$ we get
\[
L(p) < \frac{1}{1-2\norm{Q_G}_p} \big(
\frac{4(\delta+\epsilon)}{1-2\norm{Q_G}_p} \big) \leq 4n^2
(\delta_p+\epsilon) \leq 1.
\]
This proves Equation \ref{mainclaim} and the lemma.
\end{proof}

\begin{proof}[Proof of Theorem \ref{generalMhard}]
The problem of deciding whether the clique number $\omega(G)$ of a
graph $G$ is greater than  a given number~$K$ is a well-known
NP-complete problem \cite{gareyjohnson}. From Lemma
\ref{motzkinstraus2} we see that checking this is the same as
checking that for  $p = 4(e-1) n^3\ln(32 n^2)$ the integral part of
$\int_{\Delta} (Q_G)^p \d m$ is less than $K^p$. Note that the
polynomial $Q_G(\ve x)^p$ is a power of a quadratic form and can be
encoded as a SIU straight-line program of length $\Order(n^3\log
n\cdot\abs{E(G)}).$ If the computation of the integral
$\int_{\Delta} f\d m$ of a polynomial $f$ could be done in
polynomial time in the input size of $f$, we could then verify the
desired inequality in polynomial time as well.
\end{proof}

\subsection{An extension of a formula of Brion}
In this section, we obtain  several expressions   for the integrals
$\int_{\Delta} e^\ell \d m$  and $\int_\Delta \ell_1^{M_1}\cdots
  \ell_D^{M_D}{\d m}$, where $\Delta\subset \R^n$ is a
simplex and $\ell$, $\ell_j$ are linear forms on $\R^n$. The first
formula, (\ref{eq:integral-exp-big-sum}) in Lemma
\ref{th:integral-big-sum}, is obtained by elementary iterated
integration on the standard simplex. It leads to a computation of
the integral $\int_\Delta \ell_1^{M_1}\cdots
  \ell_D^{M_D}{\d m}$ in terms of the Taylor
  expansion of a certain analytic function associated to $\Delta$
  (Corollary~\ref{th:expansion}),
hence to a   proof of  the complexity result of Theorem
\ref{powersoflinear}.

In the case of one linear form $\ell$ which is  regular,  we recover
in this way the ``short formula'' of Brion as
Corollary~\ref{th:brion}. This result was first obtained
by Brion as a particular case of his theorem on polyhedra
\cite{Brion88}.

\begin{lemma}\label{th:integral-big-sum}
Let $\Delta$ be the simplex that is the convex hull of $(d+1)$
affinely independent vertices $\ve s_1, \ve s_2,\ldots,\ve s_{d+1}$
in $\R^n$, and let~$\ell$ be an arbitrary linear form on~$\R^n$. Then
\begin{equation}\label{eq:integral-exp-big-sum}
\int_\Delta e^\ell\d m = d! \vol(\Delta,\d m)
 \sum_{\ve k \in \N^{d+1}}\frac{ \langle \ell,\ve s_1\rangle ^{k_1}\cdots \langle \ell,\ve s_{d+1}\rangle^{k_{d+1}}} {(|\ve k| + d)!},
\end{equation}
where $|\ve k|=\sum_{j=1}^{d+1} k_j$.
\end{lemma}
\begin{proof}
Using an affine change of variables, it is enough to prove
(\ref{eq:integral-exp-big-sum}) when $\Delta$ is the $d$-dimensional
standard simplex $\Delta_{st}\subset \R^d$ defined by
$$
\Delta_{st}= \biggl\{\,\ve x\in\R^d : x_i\geq 0, \sum_{i=1}^d x_i\leq 1\biggr\}.
$$
The volume of $\Delta_{st}$ is equal to $\frac{1}{d!}$. In the case
of $\Delta_{st}$, the vertex  $\ve s_j$ is the  basis vector $\ve e_j$ for
$1\leq j\leq d$ and $s_{d+1}=0$. Let $\langle \ell, x\rangle
=\sum_{j=1}^d a_j x_j$. Then (\ref{eq:integral-exp-big-sum}) becomes
$$
\int_{\Delta_{st}}e^{a_1 x_1+ \dots + a_d x_d}\d \ve x= \sum_{\ve
k\in \N^d }\frac{a_1^{k_1}\cdots a_d^{k_d}}{(|\ve k| + d)!}.
$$
We prove it by induction on $d$. For $d=1$, we have
$$
\int_0^1 e^{a x}\d x= \frac{e^a -1}{a}=\sum_{k\geq
0}\frac{a^k}{(k+1)!}.
$$
Let $d>1$. We write
\begin{multline*}
\int_{\Delta_{st}}e^{a_1 x_1+ \dots + a_d x_d}\d \ve x=\\
\int_0^1 e^{a_d x_d} \biggl(\int\limits_{\substack{x_j\geq 0\\
x_1+\dots+ x_{d-1}\leq 1-x_d}} \!\!\!\!\!\!\!\!\!\!\!\!\!
e^{a_1 x_1+ \dots + a_{d-1}
  x_{d-1}}\d x_1\dots \d x_{d-1}\biggr) \d x_d.
\end{multline*}
By the induction hypothesis and an obvious change of variables, the
inner integral is equal to
$$
(1-x_d)^{d-1}\sum_{\ve k \in \N^{d-1}}(1-x_d)^{|\ve
k|}\frac{a_1^{k_1}\cdots a_{d-1}^{k_{d-1}}}{(|\ve k| + d-1)!}.
$$
The result now follows from the relation
$$
\int_0^1 \frac{(1-x)^p}{p!}e^{ax}\d x= \sum_{k\geq 0}\frac{a^k}{(k+
p+ 1)!}.
$$
\end{proof}
\begin{remark}\label{remark-after-th:integral-big-sum}
Let us replace $\ell$ by $t\ell$ in (\ref{eq:integral-exp-big-sum})
and expand in powers of $t$. We obtain the following formula.
\begin{equation}\label{eq:integral-linpower-big-sum}
\int_\Delta \ell^M \d m= d! \vol(\Delta,\d m) \frac{M!}{(M+d)!}
\sum_{\ve k \in \N^{d+1}, |\ve k| = M}\langle \ell,\ve s_1\rangle
^{k_1}\cdots \langle \ell,\ve s_{d+1}\rangle^{k_{d+1}}.
\end{equation}
This relation is a particular case of a result of Lasserre and
Avrachenkov, Proposition \ref{th:lasserre-avrachenkov}, as we will
explain in section \ref{other-formulas} below.
\end{remark}

\begin{theorem}\label{th:powers-of-linear}Let $\Delta$ be the simplex that is the convex hull of $(d+1)$
affinely independent vertices $\ve s_1, \ve s_2,\ldots,\ve s_{d+1}$
in $\R^n$.
\begin{equation}\label{eq:powers-of-linear}
 \sum_{M\in \N} t^M \frac{(M+d)!}{M!} \int_\Delta \ell^M \d m = d!\vol(\Delta, \d
 m) \frac{1}{\prod_{j=1}^{d+1}(1-t\langle \ell ,\ve s_j\rangle)}.
\end{equation}
\end{theorem}

\begin{proof}
We apply Formula (\ref{eq:integral-linpower-big-sum}). Summing up
from $M=0$ to $\infty$,  we recognize  the expansion of the
right-hand side of (\ref{eq:powers-of-linear}) into a product of
geometric series:
\begin{multline*}
 \sum_{M\in \N} t^M \frac{(M+d)!}{M!} \int_\Delta \ell^M \d m =\\
d!\vol(\Delta, \d m)\sum_{M\in \N}t^M \sum_{\ve k \in \N^{d+1}|, \ve
k| = M}\langle \ell,\ve s_1\rangle ^{k_1}\cdots \langle
\ell,\ve s_{d+1}\rangle^{k_{d+1}} .
\end{multline*}

\end{proof}
\autoref{th:powers-of-linear} has an extension   to the integration
of a product of powers of \emph{several} linear forms. The following
formula is implemented in our Maple program \texttt{duality.mpl},
see Table \ref{tab:integration-via-duality}

\begin{corollary}\label{th:expansion}Let $\ell_1,\dots,\ell_D$ be $D$ linear forms on $\R^n$.
We have the following Taylor expansion:
\begin{multline}\label{eq:expansion}
 \sum_{\ve M\in \N^D} t_1^{M_1}\cdots t_D^{M_D}\; \frac{(|\ve
  M|+d)!}{d! \vol(\Delta,\d m)}\int_\Delta \frac{\ell_1^{M_1}\cdots
  \ell_D^{M_D}}{ M_1!\dots M_D!}\d m= \\
  \frac{1}{\prod_{i=1}^{d+1}(1- t_1\langle\ell_1,\ve s_i\rangle - \cdots -t_D\langle\ell_D,\ve s_i\rangle )}.
\end{multline}
\end{corollary}
\begin{proof}
Replace $t\ell$ with $t_1\ell_1+\dots+t_D\ell_D$ in
(\ref{eq:powers-of-linear}) and take the expansion in powers
$t_1^{M_1}\cdots t_D^{M_D}$.
\end{proof}

From \autoref{th:powers-of-linear}, we obtain easily the "short
formula" of Brion, in the case of a simplex.
\begin{corollary}[Brion]\label{th:brion}
Let $\Delta$ be as in the previous theorem.  Let $\ell$
be a linear form which is regular w.r.t.~$\Delta$, i.e., $\langle
\ell,\ve s_i\rangle \neq \langle \ell,\ve s_j \rangle$ for any
  pair $i\neq j$. Then we have the following relations.
\begin{equation}\label{eq:brion-powerlinform-regular}
 \int_\Delta \ell^M \d m = d!\vol(\Delta, \d m)\frac{M!}{(M+d)!}
  \Big(\sum_{i=1}^{d+1}\frac{ \langle \ell ,\ve s_i
    \rangle^{M+d}}{\prod_{j\neq i} \langle \ell,\ve s_i-\ve s_j \rangle}\Big).
\end{equation}
\begin{equation}\label{brion}
\int_{\Delta} e^\ell \d m = d!\vol(\Delta, \d m) \sum_{i=1}^{d+1}
\frac{e^{\langle \ell ,\ve s_i\rangle}}{\prod_{j\neq i} \langle
\ell,\ve s_i-\ve s_j \rangle}.
\end{equation}
\end{corollary}
\begin{proof}
We consider the right-hand side of (\ref{eq:powers-of-linear}) as a
rational function of $t$. The poles $t= 1/\langle \ell,\ve
s_i\rangle$
 are simple precisely when $\ell$ is regular. In this case, we
obtain (\ref{eq:brion-powerlinform-regular}) by taking the expansion
into partial fractions. The second relation follows immediately by
expanding $e^\ell$.
\end{proof}
When $\ell$ is regular, Brion's formula is very short, it is a sum
of $d+1$ terms. When $\ell$ is not regular, the expansion of
(\ref{eq:powers-of-linear}) into partial fractions  leads to an
expression of  the integral  as a sum of residues.  Let
$K\subseteq\{1,\dots,d+1\}$ be an index set of the different poles
$t= 1/\langle \ell ,s_k\rangle$, and for $k\in K$ let $m_k$ denote
the order of the pole, i.e.,
\begin{displaymath}
  m_k = \#\bigl\{\, i\in\{1,\dots,d+1\} : \langle \ell ,s_i\rangle = \langle \ell ,s_k\rangle \,\bigr\}.
\end{displaymath}
With this notation, we have the following formula, which is
implemented in our Maple program \texttt{waring.mpl}, see Tables
\ref{tab:integration-of-powerlinform} and
\ref{tab:integration-by-decomposition-powerlinform}.
\begin{corollary}\label{th:residue}

\begin{multline}
  \int_{\Delta} \ell^M  \d m =\\
 d!\vol(\Delta, \d m) \frac{M!}{(M+d)!}\sum_{k\in K} \Res_{\epsilon=0} \frac{(\epsilon + \langle
    \ell, \ve s_k \rangle)^{M+d}}
  {\epsilon^{m_k} {\prod\limits_{\substack{i\in K\\ i\neq k}} {(\epsilon +
      \langle \ell, \ve s_k-\ve s_i\rangle )}^{m_i}} }
  \label{eq:integral-via-residues}
\end{multline}
\end{corollary}

\begin{remark}
 It is worth remarking that  Corollaries \ref{th:brion} and \ref{th:residue} can be seen
as a particular case of the localization theorem in equivariant
cohomology (see for instance \cite{Berline-Getzler-Vergne-1992}),
although we did  not use this fact and  instead gave a simple direct
calculation. In our situation, the variety is the complex projective
space $\C \Proj ^d$, with action of a $d$-dimensional torus, such
that the image of the moment map is the simplex $\Delta$. Brion's
formula corresponds to the generic case of a one-parameter subgroup
acting with isolated fixed points. In the degenerate  case when the
set of fixed points has components  of positive dimension, the polar
parts in  (\ref{eq:integral-via-residues}) coincide with the
contributions of the components to the localization formula.

A formula equivalent to Corollary \ref{th:residue} appears already
in \cite{Barvinok-1991} (3.2).
\end{remark}\smallbreak

\subsection{Polynomial time algorithm for polynomial functions of a fixed number of  linear forms}
\label{s:powerlinform}

\begin{proof}[Proof  of  Theorem \ref{powersoflinear}]
We now  present an  algorithm  which, given a polynomial of the
particular  form $ f(\langle \ell_1,\ve x\rangle,\cdots, \langle
\ell_D,\ve x\rangle)$ where $f$ is a polynomial depending on a fixed
number~$D$ of variables, and  $\langle \ell_j,\ve x\rangle =L_{j
1}x_1+\cdots+L_{j n} x_n$, for $j=1,\dots,D$, are  linear forms on
$\R^n$, computes its integral on a simplex, in time polynomial on
the input data. This algorithm relies on Corollary
\ref{th:expansion}.

The number of monomials of degree $M$ in $D$ variables is equal to
$\binom{M+D-1}{D-1}$. Therefore, when $D$ is fixed, the number of
monomials of degree at most $M$ in $D$ variables is $\Order(M^D)$.
When the number of variables~$D$ of a straight-line program~$\Phi$
is fixed, it is possible to compute a sparse or dense representation
of the polynomial represented by~$\Phi$ in polynomial time, by a
straight-forward execution of the program.  Indeed, all intermediate
results can be stored as sparse or dense polynomials with
$\Order(M^D)$ monomials. Since the program~$\Phi$ is
single-intermediate-use, the binary encoding size of all
coefficients of the monomials can be bounded polynomially by the
input encoding size. Thus it is enough to compute  the integral of a
monomial,
\begin{equation}\label{eq:1}
 \int _\Delta \langle \ell_1,\ve x\rangle^{M_1}\cdots
\langle \ell_D,\ve x\rangle^{M_D} \d m.
\end{equation}
From Corollary \ref{th:expansion}, it follows  that
$$
 \frac{(|\ve
  M|+d)!}{d! \vol(\Delta,\d m)}\int_\Delta \frac{\ell_1^{M_1}\cdots
  \ell_D^{M_D}}{ M_1!\cdots M_D!}\d m
  $$
  is the coefficient of $ t_1^{M_1}\cdots
  t_D^{M_D}$ in the Taylor expansion of
  $$
\frac{1}{\prod_{i=1}^{d+1}(1- t_1\langle\ell_1,\ve s_i\rangle -
\cdots -t_D\langle\ell_D,\ve s_i\rangle )}.
  $$
  Since $D$
is fixed,  this coefficient  can be computed in time polynomial with
respect to $\ve M $ and the input data, by multiplying  series
truncated at degree $|\ve M|$, as explained in Lemma
\ref{truncated-product}.

  Finally,
  $\vol(\Delta, \d m)$ needs to be computed.
  If $\Delta=\conv\{\ve
  s_1,\dots,\ve s_{d+1}\}$ is full-dimensional ($d = n$), we can do so by
  computing the determinant of the matrix formed by difference vectors of the vertices:
  \begin{displaymath}
    \vol(\Delta, \d m)=\frac1{n!} \abs{\det(\ve s_1-\ve s_{n+1},
      \dots, \ve s_{n}-\ve s_{n+1})}.
  \end{displaymath}
  If $\Delta$ is lower-dimensional, we first compute a basis~$B\in\Z^{n\times
    d}$ of the intersection lattice~$\Lambda = \lin(\Delta)\cap\Z^n$.  This
  can be done in polynomial time by applying an efficient algorithm for
  computing the Hermite normal form \cite{kannan-bachem:79}.  Then we express
  each difference vector $\ve v_i = \ve s_i-\ve s_{d+1} \in \lin(\Delta)$ for $i=1,\dots,d$
  using the basis~$B$ as $\ve v_i = B \ve v_i'$, where $\ve v_i\in\Q^d$.  We obtain
  \begin{displaymath}
    \vol(\Delta, \d m) = \frac1{d!} \abs{\det(\ve v_1',\dots, \ve v_d')},
  \end{displaymath}
  thus the volume computation is reduced to the calculation of a determinant.
This finishes the proof of Theorem \ref{powersoflinear}. \end{proof}

\subsection{Polynomial time algorithms for polynomials of fixed
degree}\label{s:fixed_degree}

In the present section, we  assume that the total degree of the
input polynomial $f$ we wish to integrate is a constant $M$.
\begin{proof}[Proof of Corollary \ref{fixMpolytime}]
First of all, when the formal degree~$M$ of a straight-line
program~$\Phi$ is fixed, it is possible to compute a sparse or dense
representation of the polynomial represented by~$\Phi$ in polynomial
time, by a straight-forward execution of the program.  Indeed, all
intermediate results can be stored as sparse or dense polynomials
with $\Order(n^M)$ monomials.  Since the program~$\Phi$ is
single-intermediate-use, the binary encoding size of all
coefficients of the monomials can be bounded polynomially by the
input encoding size.

Now, the key observation is that a monomial of degree at most~$M$ depends
effectively on $D\leq M$ variables $x_{i_1}$, \dots, $x_{i_D}$, thus it is of
the form
$$
\ell_1^{M_1}\cdots \ell_D^{M_D},
$$
where the linear forms $\ell_j(\ve x) = x_{i_j}$ are the coordinates
that effectively appear in the monomial. Thus,
\autoref{fixMpolytime} follows immediately from
\autoref{powersoflinear}. This method is implemented in our Maple
program \texttt{duality.mpl}, see Tables
\ref{tab:integration-via-duality} and
\ref{tab:integration-dense-poly-via-duality}.
\end{proof}

\begin{remark}\label{method_with_Taylor}
The relations in Corollary \ref{th:brion} can be interpreted as
equalities between meromorphic functions of $\ell$. The right-hand
side is a sum of meromorphic functions whose poles cancel out, so
that the sum is actually analytic. We derive from this another
polynomial time algorithm for computating  the integral
$$
 \int_{\Delta} x_1^{m_1}\cdots  x_d^{m_d} \d m.
$$
More precisely, let us write $\langle\ell,\ve x \rangle= y_1 x_1 +
\cdots + y_d x_d$. Then the integral $ \int_{\Delta} x_1^{m_1}\cdots
x_d^{m_d} \d m$ is the coefficient of $ \frac{y_1^{m_1}\ldots
y_d^{m_d}}{m_1!\cdots m_d!}$ in the  Taylor expansion of
$\int_{\Delta}e^{y_1 x_1 + \ldots y_d x_d}\d m$. We compute it by
taking the expansion of each of the terms of the right hand-side of
Equation (\ref{brion}) into an iterated Laurent series with respect
to the variables  $x_1,\ldots, x_d$. This method is implemented in
our Maple program \texttt{iterated-laurent.mpl}, see Tables
\ref{tab:integration-dense-poly-via-iterated} and
\ref{tab:integration-few-effective-via-iterated}.
\end{remark}

In the following, we give another proof of Corollary \ref{fixMpolytime}, based
on decompositions of polynomials
as sums of powers of linear forms.

\begin{proof}[Alternative proof of \autoref{fixMpolytime}]\label{proof-by-waring}

From  Corollaries \ref{th:brion} and \ref{th:residue}, we derive
another efficient algorithm, as follows. The key idea now is that
one can decompose the polynomial $f$ as a sum
$f:=\sum_{\ell}c_{\ell} \ell_j^M$ with at most $2^M$ terms in the
sum. We use the well-known identity
\begin{multline}
  x_1^{M_1}x_2^{M_2}\cdots x_n^{M_n}\\
  = \frac{1}{|\ve M|!} \sum_{0\leq p_i\leq M_i}(-1)^{|\ve M|-(p_1+\cdots+p_n)}
  \binom{M_1}{p_1}\cdots \binom{M_n}{p_n}(p_1 x_1+\cdots+p_n x_n)^{|\ve M|},
  \label{eq:decomp-powerlinform}
\end{multline}
where $|\ve M| = M_1+\cdots+M_n \leq M$.

In the  implementation of  this method, we may group together
proportional linear forms. The number $F(n,M)$ of primitive vectors
$(p_1,\dots, p_n)$ which appear in the decomposition of a polynomial
of total degree $\leq M$ is given by the following closed formula.
\footnote{Lemma \ref{number_of_primitive} was kindly supplied by
Christophe Margerin.}.

\begin{lemma}\label{number_of_primitive}
Let
\begin{multline*}F(n,M)=\\ \card(\{(p_1,\dots, p_n)\in \N^n, \gcd(p_1,\dots,
p_n)= 1, 1\leq \sum_i p_i\leq  M  \}).\end{multline*}

Then  \begin{equation}\label{eq:number_of_primitive} F(n,M)=
\sum_{d=1}^M \mu(d)(\binom{n+[\frac{M}{d}]}{n}-1),
\end{equation}
where $\mu(d)$ is the M\"{o}bius function.

When $M$ is fixed and $n\to \infty$ we have
$$F(n,M)=\frac{n^M}{M!}+ \Order(n^{M-1}).$$
\end{lemma}
\begin{proof}
Let $G(n,M)= \card(\{(p_1,\dots, p_n)\in \N^n,  1\leq \sum p_i\leq M
\})$. By grouping together the vectors $(p_1,\dots, p_n)$ with a
given $\gcd$ $d$, we obtain
$$
G(n,M)=\sum_{d=1}^M F(n, [\frac{M}{d}]).
$$
Moreover the number of all integral vectors $(p_1,\dots, p_n)\in
\N^n$ such that $\sum_i p_i\leq  M $ is equal to the binomial
coefficient $\binom{n+M}{n}$. When we omit the zero-vector we obtain
$$
G(n,M)= \binom{n+M}{n}-1.
$$
Then we obtain (\ref{eq:number_of_primitive}) by applying the second
form of M\"{o}bius inversion formula.  The asymptotics follow easily
from (\ref{eq:number_of_primitive}).
\end{proof}
Thus Formula (\ref{eq:decomp-powerlinform}), together with Corollary
\ref{th:residue},  give another polynomial time algorithm for
integrating a polynomial of fixed degree. It is implemented in our
Maple program \texttt{waring.mpl}, see Tables
\ref{tab:integration-by-decomposition-powerlinform} and
\ref{tab:integration-dense-poly-by-decomposition-powerlinform}.
\end{proof}

 The problem of finding a decomposition
 with the smallest possible number of summands
 is known as the \emph{polynomial Waring problem}.  Alexander and Hirschowitz
 solved the generic problem (see
\cite{alexanderhirschowitz}, and \cite{brambillaottaviani} for an
extensive survey).
\begin{theorem}\label{th:alexanderhirschowitz} The smallest integer
$r(M, n)$ such that a generic homogeneous polynomial  of degree~$M$
in $n$ variables is expressible as the sum of $r(M,n)$ $M$-th powers
of linear forms is given by
$$r(M, n)=\bigg\lceil \frac{{{n+M-1} \choose M}}{n} \bigg\rceil,$$
 with the exception of the cases $r(3,
5) = 8$, $r(4, 3) = 6$, $r(4, 4) = 10$, $r(4, 5) = 15$, and $M = 2$,
where $r(2, n) = n$.
\end{theorem}

 An algorithm for
decomposing a given polynomial into the smallest possible number of
powers of linear forms can be found in \cite{brachat}.

  In the extreme case, when the polynomial~$f$ happens to be the
power of one linear form $\ell$, one should certainly avoid applying the above
decomposition formula to each of the monomials of~$f$.  We remark that, when
the degree is fixed, we can decide in polynomial time
whether a polynomial $f$, given in sparse or dense monomial representation, is
a power of a linear form $\ell$, and, if so, construct such a linear form.

\section{Other  algorithms for integration and extensions to other polytopes}
\label{otheralgsextensions}

We conclude with a discussion of how to extend integration to other polytopes and
a review of the complexity of other methods to integrate polynomials over polytopes.

\subsection{A formula of Lasserre--Avrachenkov}
\label{s:lasserre-avrachenkov}
Another nice formula is the Lasserre--Avrachenkov formula for the
integration of a homogeneous polynomial
\cite{lasserre-avrachenkov2001} on a simplex. As we explain below,
this yields a polynomial-time algorithm for the problem of
integrating a polynomial of fixed degree over a polytope in varying
dimension, thus providing an alternative proof of
\autoref{fixMpolytime}.

\begin{proposition}[\cite{lasserre-avrachenkov2001}]
  \label{th:lasserre-avrachenkov}
Let $H$ be a symmetric  multilinear form defined on $(\R^d)^M$. Let
$\ve s_1,\ve s_2,\dots,\ve s_{d+1}$ be the vertices of a
$d$-dimensional simplex $\Delta$. Then one has
\begin{equation}\label{eq:multilin-integral}
\int_\Delta H( \ve x, \ve x,\dots, \ve x) d \ve
x=\frac{\vol(\Delta)}{\binom{M+d}{M}} \sum_{1 \leq i_1 \leq  i_2
\leq \cdots  i_M \leq d+1} \!\!\!\!\! H( \ve s_{i_1}, \ve
s_{i_2},\dots, \ve s_{i_M}).
\end{equation}
\end{proposition}
\begin{remark}
By reindexing the summation in (\ref{eq:multilin-integral}), as $H$
is symmetric, we obtain
\begin{eqnarray}\label{eq:multilin-integral-2}
\nonumber \int_\Delta H( \ve x, \ve x,\dots, \ve x) d \ve x
=&&\\
\frac{\vol(\Delta)}{\binom{M+d}{M}}\sum_{k_1+\cdots +k_{d+1}=M}&&
H(\ve s_1,\dots, \ve s_1,\dots,\ve s_{d+1},\dots,\ve s_{d+1}),
\end{eqnarray}
where $\ve s_1$ is repeated $k_1$ times, $\ve s_2$ is repeated $k_2$
times, etc. When $H$ is of the form  $H=\prod_{i=1}^M \langle
\ell,\ve x_i\rangle$, for a single linear form $\ell$, then
(\ref{eq:multilin-integral-2}) coincides with Formula
(\ref{eq:integral-linpower-big-sum}) in
\autoref{remark-after-th:integral-big-sum}.
\end{remark}
Now any polynomial $f$ which is homogeneous of degree $M$ can be
written as  $f(\ve x)=H_f( \ve x, \ve x,\dots, \ve x)$ for a unique
multilinear form $H_f$. If $f=\ell^M$ then $H_f=\prod_{i=1}^M
\langle \ell,\ve x_i\rangle$. Thus for fixed $M$ the computation of
$H_f$ can be done
 by decomposing $f$ into a linear combination of powers of
linear forms, as we did in the proof of Corollary
\ref{fixMpolytime}.  Alternatively one can use  the well-known
polarization formula,
\begin{equation}
  H_f(\ve x_1,\dots,\ve x_M)=\frac{1}{2_{\mathstrut}^M M!} \sum_{\substack{\ve\epsilon \in{\{\pm
        1\}}^M}} \!\! \epsilon_1\epsilon_2\cdots\epsilon_M
        f\Bigl(\sum_{i=1}^M
  \epsilon_i \ve x_i\Bigr),\label{eq:polarization}
\end{equation}
Thus from \eqref{eq:multilin-integral} we get the following
corollary.
\begin{corollary}\label{cor:integration-via-polarization}
  Let $f$ be a homogeneous polynomial of degree~$M$ in $d$ variables, and let $\ve s_1,\ve
  s_2,\dots,\ve s_{d+1}$ be the vertices of a $d$-dimensional simplex $\Delta$.
  Then
  \begin{multline}
    \label{eq:integration-via-polarization}
      \int_\Delta f(\ve y) \d\ve y =\\
      \frac{\vol(\Delta)}{2_{\mathstrut}^M
        M!\,\binom{M+d}{M}} \sum_{1 \leq i_1 \leq i_2 \leq \cdots \leq i_M \leq d+1}
      \sum_{\substack{\ve\epsilon \in{\{\pm 1\}}^M}} \!\!
      \epsilon_1\epsilon_2\cdots\epsilon_M f\Bigl(\sum_{k=1}^M \epsilon_k \ve
      s_{i_k}\Bigr).
  \end{multline}
\end{corollary}
We remark that when we fix the degree~$M$ of the homogeneous
polynomial~$f$, the length of the polarization formula (thus the
length of the second sum in \eqref{eq:integration-via-polarization}
is a constant.  The length of the first sum in
\eqref{eq:integration-via-polarization} is $\Order(n^M)$.  Thus, for
fixed degree in varying dimension, we obtain another polynomial-time
algorithm for integrating over a simplex.




\subsection{Traditional conversion of the integral as iterated univariate integrals}

 Let
$P\subseteq\R^d$ be a full-dimensional polytope and $f$ a
polynomial. The traditional method we teach our calculus students to
compute multivariate integrals over a bounded region requires them
to write the integral $\int_P f \d m$ is a sum of sequences of
one-dimensional integrals
\begin{equation}\label{iteratedonevar}
\sum_{j=1}^K \int_{a_{1j}}^{b_{1j}} \int_{a_{2j}}^{b_{2j}} \cdots
\int_{a_{dj}}^{b_{dj}} f \d x_{i_1}\d x_{i_2} \dots \d x_{i_d}
\end{equation}
for which we know the limits of integration $a_{ij},b_{ij}$
explicitly. The problem of finding the limits of integration and the
sum has interesting complexity related to the well-known
Fourier-Motzkin elimination method (see Chapter One in
\cite{zieglerpolybook} for a short introduction).

Given a system of linear inequalities $A\ve x \leq \ve b$,
describing a polytope $P \subset \R^d$, Fourier--Motzkin elimination
is an algorithm that computes a new larger system of inequalities
$\hat{A}\ve x \leq \ve{\hat{b}}$ with the property that those
inequalities that do not contain the variable $x_d$ describe the
projection of $P$ into the hyperplane $x_d=0$. We will not explain
the details, but Fourier-Motzkin elimination is quite similar to
Gaussian elimination in the sense that the main operations necessary
to eliminate the last variable $x_d$ require to rearrange, scale,
and add rows of the matrix $(A,\ve b)$, but unlike Gaussian
elimination, new inequalities are added to the system.

It was first observed by Schechter \cite{schechter} that
Fourier-Motzkin elimination provides a way to generate the
traditional iterated integrals. More precisely, let us call $P_d$
the projection of $P$ into the the hyperplane $x_d=0$. Clearly when
integrating over a polytopal region we expect that the limits of
integration will be affine functions. From the output of
Fourier-Motzkin $\hat{A}\ve x \leq \ve{\hat{b}}$, we have $\ve x\in
P$ if and only if $(x_1,\dots,x_{d-1}) \in P_d$ and for the first
$k+r$ inequalites of the system
$$ x_d \leq \hat{b}_i-\sum_{j=1}^{d-1} \hat{a}_{ij} x_j=A_i^{\text{u}}(x_1,\dots,x_{d-1})$$ for $i=1,\dots k$
as well as
$$ x_d \geq \hat{b}_{k+i}-\sum_{j=1}^{d-1} \hat{a}_{{k+i}j} x_j,=A_i^{\text{l}}(x_1,\dots,x_{d-1})$$
for $i=1,\dots,r$. Then, if we define
$$m(x_1,\dots, x_d)=\max \{A_j^{\text{l}}(x_1,\dots,x_{d-1}), \  j=1,\dots,r \}$$
and
$$M(x_1,\dots, x_d)=\min\{A_j^{\text{u}}(x_1,\dots,x_{d-1}), \  j=1,\dots,r \},$$
we can write $$\int_P f(\ve x) \d m= \int_{P_d} \int_{m}^M f(\ve x)
\d x_1\d x_2\cdots \d x_d$$ Finally the convex polytope $P_d$ can be
decomposed into polyhedral regions where the functions $m,M$ become
simply affine functions from among the list. Since the integral is
additive we get an expression
$$\int_P f(\ve x) \d m=\sum_{i,j} \int_{P_d^{ij}} \int_{A_j^{\text{l}}}^
{A_j^{\text{u}}} f(\ve x) \d x_1 \d x_2\cdots \d x_d.$$ Finally by
repeating the elimination of variables we recover the full iterated
list in~\eqref{iteratedonevar}. As it was observed in
\cite{schechter}, this algorithm is unfortunately not efficient
because the iterated Fourier-Motzkin elimination procedure can
produce exponentially many inequalities for the description of the
projection (when the dimension~$d$ varies). Thus the number of
summands considered can in fact grow exponentially.

\subsection{Two formulas for integral
computation}\label{other-formulas}

We would like to review two formulas that are nice and could speed
up computation in particular cases although they do not seem to
yield efficient algorithms just on their own.

First, one may reduce the computation of $\int_P f \d m$ to
integrals on the facets of $P$, by applying  Stokes formula. We must
be careful to use a rational Lebesgue measure on each facet.  As
shown in (\cite{integralsbarvinok}), we have the following result.
\begin{theorem}\label{th:stokes}
Let $\{F_i\}_{i=1,\dots,m}$ be the set of facets of a
full-dimensional polytope $P\subseteq\R^n$. For each $i$, let  $\ve
n_i$ be a rational vector which is transverse to the facet $F_i$ and
pointing outwards $P$ and let  $\d \mu_i$ be the Lebesgue measure on
the affine hull of  $F_i$ which is defined by  contracting  the
standard volume form of $\R^n$ with $\ve n_i$. Then
$$I_P(\ve a) = \int_P e^{\langle \ve a,\ve x \rangle} \d\ve x=\frac{1}{\langle \ve a,\ve y \rangle}
\sum_{i=1}^m \langle  \ve n_i,\ve y \rangle \int_{F_i} e^{\langle
\ve a,\ve x
  \rangle} \d\mu_i$$
for all $\ve a\in\C^n$ and $\ve y \in {\R}^n$ such that $\langle \ve
y,\ve a \rangle \not=0$.
\end{theorem}
It is clear that, by considering the expansion of the analytic
function $\int_P e^{\langle \ve a, \ve x \rangle} \d \ve x$, we can
again obtain an analogous result for polynomials.  An alternative
proof was provided by \cite{lasserre-integrapoly}. The above
theorem, however, does not necessarily reduce the computational
burden because, depending on the representation of the polytope, the
number of facets can be large and also the facets themselves can be
complicated polytopes.  Yet, together with our results we obtain the
following corollary for two special cases.
\begin{corollary}
  There is a polynomial-time algorithm for the following problem.
  Input:
  \begin{inputlist}
  \item the dimension $n\in\N$ in unary encoding,
  \item a list of rational vectors in binary encoding, namely
    \begin{enumerate}[\rm(i)]
    \item either vectors $(\ve h_1, h_{1,0}), \dots, (\ve h_m,
      h_{m,0})\in\Q^{n+1}$ that describe the facet-defining
      inequalities $\langle\ve h_i, \ve x\rangle \leq h_{i,0}$ of a \emph{simplicial}
      full-dimensional  rational polytope~$P$,
    \item or vectors $\ve s_1, \dots, \ve s_N\in\Q^n$ that are the vertices
      of a \emph{simple} full-dimensional rational polytope~$P$,
    \end{enumerate}
  \item a rational vector~$\ve a\in\Q^n$ in binary encoding,
  \item an exponent $M\in\N$ in unary encoding.
  \end{inputlist}
  Output, in binary encoding,
  \begin{outputlist}
  \item the rational number
    \begin{displaymath}
      \int _P f(\ve x) \d m \quad\text{where}\quad
      f(\ve x) = \langle \ve a, \ve x \rangle^M
    \end{displaymath}
    and where $\d m$ is the standard Lebesgue measure on~$\R^n$.
  \end{outputlist}
\end{corollary}
\begin{proof}
  In the case (i) of simplicial polytopes~$P$ given by facet-defining
  inequalities, we can use linear programming to compute in polynomial time a
  $V$-representation for each simplex $F_i$ that is a facet of~$P$.
  By applying  \autoref{th:stokes}
  with $ t\ve a$ in place of $\ve a$ and extracting the coefficient of $t^M$ in the Taylor
  expansion of the analytic function $t\mapsto I_P(t\ve a)$,
  we obtain the formula
  $$\int_P \langle \ve a,\ve x \rangle^M \d\ve x
  =\frac{1}{(M+1)\langle \ve y, \ve a \rangle}
  \sum_{i=1}^m \langle \ve y, \ve n_i \rangle \int_{F_i} \langle\ve a,\ve
  x\rangle^{M +1} \d\mu_i,$$
  which holds for all~$\ve y\in\R^n$ with $\langle\ve y,\ve a\rangle\neq0$.
  It is known that a suitable~$\ve y\in\Q^n$ can be constructed in polynomial time.
  The integrals on the right-hand side can now be evaluated in polynomial time
  using \autoref{powersoflinear}.

  In the case (ii) of simple polytopes~$P$ given by their vertices, we make
  use of the fact that a variant of Brion's formula (\ref{brion}) actually holds for arbitrary
  rational polytopes. For  a
  simple polytope $P$, it takes the following form.
  \begin{equation}\label{Brion-simple-polytope}
\int_P \ell^M \d x =\frac{M!}{(M+n)!} \sum_{i=1}^N \Delta_i
  \frac{ \langle \ell ,\ve s_i
    \rangle^{M+n}}{\prod_{\ve s_j \in N(\ve s_i)} \langle \ell,\ve s_i-\ve
    s_j \rangle},
    \end{equation}
  where $N(\ve s_i)$ denotes the set of vertices adjacent to~$\ve s_i$
  in~$P$, and $\Delta_i= \mathopen|\det(\ve s_i-\ve
    s_j)_{j\in N(\ve s_i)}\mathclose|$.
The right-hand side is a sum of rational functions of $\ell$, where
the denominators cancel out so that the sum is actually polynomial.
If $\ell$ is regular, that is to say $ \langle \ell,\ve s_i-\ve
    s_j \rangle \neq 0$ for any $i$ and $j\in N(\ve s_i)$, then the
    integral can be computed by (\ref{Brion-simple-polytope}) which is a very short formula.
However it becomes difficult to extend the method which we used in
the case of a simplex. Instead, we can do a perturbation. In
(\ref{Brion-simple-polytope}), we replace $\ell$ by $\ell + \epsilon
\ell'$, where $ \ell'$ is such that $\ell + \epsilon \ell'$ is
regular for $\epsilon \neq 0$. The algorithm for choosing $\ell'$ is
bounded polynomially. Then we do expansions in powers of $\epsilon $
as explained in Lemma \ref{truncated-product}.
 \end{proof}

\subsection{Triangulation of arbitrary polytopes}
It is well-known that any convex polytope can be triangulated into
finitely many simplices.  Thus we can use our result
to extend the integration of polynomials over any convex polytope.
The complexity of doing it this way will directly depend
on the number of simplices in a triangulation. This raises the issue
of finding the smallest triangulation possible of a given
polytope. Unfortunately this problem was proved to be NP-hard even for
fixed dimension three (see  \cite{triangbook}).
Thus it is in general not a good idea to spend time finding the smallest
triangulation possible. A cautionary remark is that one can
naively assume that triangulations help for non-convex polyhedral
regions, while in reality it does not because there exist nonconvex polyhedra
that are not triangulable unless one adds new points. Deciding how many
new points are necessary is an NP-hard problem \cite{triangbook}.


\section{Implementation and  computational experiments}
\label{s:computations}

We have written Maple programs to perform some initial experiments
with the three methods described in Section \ref{s:fixed_degree}.
The programs are available at
\cite{deloeraetal:integration-paper:accompanying-programs}.\footnote{All
  algorithms are implemented in the files \texttt{waring.mpl}, \texttt{iterated\_laurent.mpl}, \texttt{duality.mpl}; all tables with random examples
  are created using procedures in \texttt{examples.mpl} and
  \texttt{tables.mpl}.}

\subsection{Integration of a power of a linear form and decomposition of polynomials into powers of linear forms }

\subsubsection{Decomposition of polynomials into powers of linear
forms}
\autoref{tab:decomposition-powerlinform} shows the number $F(n,M)$
of primitive linear forms $(p_1,\dots, p_n)$ which may appear in the
decomposition ~\eqref{eq:decomp-powerlinform} of a polynomial of
total degree $\leq M$. This number is computed using the closed
formula \eqref{eq:number_of_primitive}.

\begin{table}[ht]
  \caption{Decomposition of polynomials into powers of linear forms}
  \label{tab:decomposition-powerlinform}
  \small
  \begin{center}
$\begin{array}{r*{8}r}
\toprule
 & \multicolumn{8}{c}{\text{Degree~$M$}}\\
\cmidrule{2-9}
n & 1 & 2 & 5 & 10 & 20 & 30 & 40 & 50\\
\midrule
2 & 2
 & 3
 & 11
 & 33
 & 129
 & 279
 & 491
 & 775
\\
3 & 3
 & 6
 & 40
 & 205
 & 1381
 & 4306
 & 9880
 & 18970
\\
4 & 4
 & 10
 & 103
 & 831
 & 9373
 & 41373
 & 122349
 & 286893
\\
5 & 5
 & 15
 & 221
 & 2681
 & 49586
 & 305836
 & 1.2\cdot10^{6}
 & 3.3\cdot10^{6}
\\
8 & 8
 & 36
 & 1226
 & 42271
 & 3.1\cdot10^{6}
 & 4.8\cdot10^{7}
 & 3.7\cdot10^{8}
 & 1.9\cdot10^{9}
\\
10 & 10
 & 55
 & 2917
 & 181413
 & 3.0\cdot10^{7}
 & 8.4\cdot10^{8}
 & 1.0\cdot10^{10}
 & 7.5\cdot10^{10}
\\
15 & 15
 & 120
 & 15338
 & 3.3\cdot10^{6}
 & 3.2\cdot10^{9}
 & 3.4\cdot10^{11}
 & 1.2\cdot10^{13}
 & 2.1\cdot10^{14}
\\
20 & 20
 & 210
 & 52859
 & 3.0\cdot10^{7}
 & 1.4\cdot10^{11}
 & 4.7\cdot10^{13}
 & 4.2\cdot10^{15}
 & 1.6\cdot10^{17}
\\
30 & 30
 & 465
 & 324076
 & 8.5\cdot10^{8}
 & 4.7\cdot10^{13}
 & 1.2\cdot10^{17}
 & 5.5\cdot10^{19}
 & 8.9\cdot10^{21}
\\
40 & 40
 & 820
 & 1.2\cdot10^{6}
 & 1.0\cdot10^{10}
 & 4.2\cdot10^{15}
 & 5.5\cdot10^{19}
 & 1.1\cdot10^{23}
 & 6.0\cdot10^{25}
\\
50 & 50
 & 1275
 & 3.5\cdot10^{6}
 & 7.5\cdot10^{10}
 & 1.6\cdot10^{17}
 & 8.9\cdot10^{21}
 & 6.0\cdot10^{25}
 & 1.0\cdot10^{29}
\\
\bottomrule
\end{array}$
  \end{center}
\end{table}

\newcommand\nd{\hphantom{.0}}

\clearpage
\subsubsection{Integration of a power of a linear form over a simplex}

We have written a Maple program which implements the method of
Corollary \ref{th:residue} for the efficient integration of a power
of one linear form over a simplex, $\int_\Delta \ell^M \d
m$.\footnote{The
  integration is done by the Maple procedure
  \textit{integral\_power\_linear\_form} in \texttt{waring.mpl}.}
In a computational experiment,
for a given dimension~$n$ and degree~$M$, we picked random full-dimensional
simplices~$\Delta$ and random linear forms~$\ell$ and used the Maple program to compute the
integral.  \autoref{tab:integration-of-powerlinform} shows the computation
times.\footnote{All experiments were done with Maple~12 on Sun Fire V440 machines
  with UltraSPARC-IIIi processors running at 1.6\,GHz.  The computation times
  are given in CPU seconds. All experiments were subject to a time limit of
  600 seconds per example.}

\begin{table}[ht]
  \caption{Integration of powers of linear forms over simplices}
  \label{tab:integration-of-powerlinform}
  \begin{center}
  \small

\begin{tabular}{r*{7}r}
\toprule
 & \multicolumn{7}{c}{Degree~$M$}\\
\cmidrule(l){2-8}
$n$ & 2 & 10 & 20 & 50 & 100 & 300 & 1000\\
\midrule
10 &  0.0 &  0.0 &  0.0 &  0.1 &  0.0 &  0.0 &  0.0\\
20 &  0.1 &  0.2 &  0.1 &  0.2 &  0.1 &  0.2 &  0.2\\
50 &     1.0 &   1.4 &   1.4 &   1.6 &   1.6 &   1.6 &   1.7\\
100 &   5.1 &   8.4 &   8.7 &   9.2 &   9.5 &    10.0 &    11.0\\
200 &    36\nd &    71\nd &    84\nd &    88\nd &    97\nd & 110\nd & 120\nd\\
300 & 150\nd & 320\nd & 400\nd & 470\nd & 520\nd & 530\nd & 
\\
400 & 500\nd & 
\\
1000 & 
\\
\bottomrule
\end{tabular}
\end{center}
\end{table}

\subsubsection{Integration of a  monomial over a simplex by decomposition as sum of  powers of linear forms}

Next, we tested the  algorithm which  computes the integral of a
monomial $\ve x^{\ve M}$ over a simplex~$\Delta$, by decomposing it
as a sum of  powers of linear forms. This algorithm was discussed in
\autoref{s:fixed_degree}. In our experiments, for given dimension
$n$ and total degree $M$, we picked 50 combinations of a random
simplex~$\Delta$ of dimension $n$ and a random  exponent vector~$\ve
M=(M_1,\ldots,M_n)$ with $ \sum_{i=1}^n M_i= M$.

First we decompose a given monomial into a sum of powers of linear
forms, then we integrate each summand using the Maple procedure
discussed above. \footnote{This method is implemented in the Maple
procedure
  \textit{integral\_via\_waring} in \texttt{waring.mpl}.}
\autoref{tab:integration-by-decomposition-powerlinform} shows the minimum,
average, and maximum computation
times.

\begin{table}[t]
  \caption{Integration of a random monomial of prescribed degree by
    decomposition into a sum of powers of linear forms}
  \label{tab:integration-by-decomposition-powerlinform}
  \tiny
  \begin{center}
\begin{tabular}{r*{11}r@{}r}
\toprule
 & \multicolumn{12}{c}{Degree}\\
\cmidrule{2-13}
$n$ & 1 & 2 & 5 & 10 & 20 & 30 & 40 & 50 & 100 & 200 & 300 & \\
\midrule
2 & \minavgmax{      0}{      0}{      0}
 & \minavgmax{      0}{      0}{    0.1}
 & \minavgmax{      0}{      0}{    0.1}
 & \minavgmax{      0}{    0.1}{    0.1}
 & \minavgmax{      0}{    0.2}{    0.4}
 & \minavgmax{      0}{    0.5}{    0.8}
 & \minavgmax{    0.1}{    1.0}{    1.9}
 & \minavgmax{    0.3}{    1.5}{    2.6}
 & \minavgmax{    0.8}{    8.3}{     12}
 & \minavgmax{    6.2}{     38}{     59}
 & \minavgmax{    0.3}{     90}{    174}
 & 
\\[7ex]
3 & \minavgmax{      0}{      0}{      0}
 & \minavgmax{      0}{      0}{    0.2}
 & \minavgmax{      0}{      0}{    0.1}
 & \minavgmax{      0}{    0.2}{    0.4}
 & \minavgmax{      0}{    1.1}{    2.2}
 & \minavgmax{    0.4}{    3.0}{    7.0}
 & \minavgmax{    0.5}{    7.4}{     19}
 & \minavgmax{    0.9}{     17}{     37}
 & \minavgmax{    4.6}{    173}{    440}
 & 
\\[7ex]
4 & \minavgmax{      0}{      0}{      0}
 & \minavgmax{      0}{      0}{    0.2}
 & \minavgmax{      0}{    0.1}{    0.2}
 & \minavgmax{    0.1}{    0.4}{    0.9}
 & \minavgmax{    0.6}{    3.6}{    8.7}
 & \minavgmax{    1.5}{     15}{     44}
 & \minavgmax{    4.8}{     52}{    149}
 & \minavgmax{    9.7}{    135}{    404}
 & 
\\[7ex]
5 & \minavgmax{      0}{      0}{      0}
 & \minavgmax{      0}{      0}{    0.2}
 & \minavgmax{      0}{    0.1}{    0.3}
 & \minavgmax{    0.1}{    0.7}{    1.9}
 & \minavgmax{    0.3}{    8.8}{     27}
 & \minavgmax{    4.5}{     48}{    195}
 & 
\\[7ex]
6 & \minavgmax{      0}{      0}{    0.2}
 & \minavgmax{      0}{      0}{    0.1}
 & \minavgmax{      0}{    0.2}{    0.4}
 & \minavgmax{    0.2}{    1.3}{    2.7}
 & \minavgmax{    1.3}{     24}{     74}
 & \minavgmax{    8.0}{    144}{    544}
 & 
\\[7ex]
7 & \minavgmax{      0}{      0}{    0.2}
 & \minavgmax{      0}{      0}{    0.1}
 & \minavgmax{      0}{    0.3}{    0.6}
 & \minavgmax{    0.5}{    2.1}{    5.0}
 & \minavgmax{    5.9}{     53}{    152}
 & 
\\[7ex]
8 & \minavgmax{      0}{      0}{    0.2}
 & \minavgmax{      0}{      0}{    0.2}
 & \minavgmax{    0.1}{    0.3}{    0.6}
 & \minavgmax{    0.4}{    3.2}{    8.5}
 & \minavgmax{     11}{     72}{    216}
 & 
\\[7ex]
10 & \minavgmax{      0}{      0}{    0.3}
 & \minavgmax{      0}{    0.1}{    0.2}
 & \minavgmax{    0.2}{    0.4}{    0.8}
 & \minavgmax{    1.5}{    6.1}{     12}
 & 
\\[7ex]
15 & \minavgmax{      0}{    0.1}{    0.2}
 & \minavgmax{      0}{    0.1}{    0.3}
 & \minavgmax{    0.3}{    1.2}{    1.8}
 & \minavgmax{    3.8}{     17}{     41}
 & 
\\[7ex]
20 & \minavgmax{    0.1}{    0.1}{    0.2}
 & \minavgmax{    0.1}{    0.3}{    0.4}
 & \minavgmax{    0.6}{    2.2}{    2.9}
 & \minavgmax{    4.4}{     41}{     73}
 & 
\\[7ex]
30 & \minavgmax{    0.1}{    0.2}{    0.3}
 & \minavgmax{    0.2}{    0.5}{    0.6}
 & \minavgmax{    2.7}{    5.1}{    6.8}
 & \minavgmax{     37}{    106}{    170}
 & 
\\[7ex]
40 & \minavgmax{    0.3}{    0.4}{    0.6}
 & \minavgmax{    0.3}{    1.1}{    1.3}
 & \minavgmax{    5.2}{     10}{     12}
 & \minavgmax{     93}{    242}{    414}
 & 
\\[7ex]
50 & \minavgmax{    0.5}{    0.6}{    0.8}
 & \minavgmax{    0.7}{    1.8}{    2.0}
 & \minavgmax{    8.2}{     17}{     20}
 & 
\\
\bottomrule
\end{tabular}
  \end{center}
\end{table}

\clearpage
\subsection{Integration of a monomial, using iterated
Laurent series.} In this section, we  test   the implementation of
the method of iterated Laurent expansion described in Remark
\ref{method_with_Taylor} of \autoref{s:fixed_degree}.\footnote{This
method is implemented in the
  Maple procedure \textit{integral\_via\_iterated}, defined in the file
  \texttt{iterated\_laurent.mpl}.}

Table~\ref{tab:integration-via-iterated} shows the results.

\begin{table}[ht]
  \caption{Integration of a random monomial of prescribed degree using
    iterated Laurent series.}
  \label{tab:integration-via-iterated}
  \tiny
  \begin{center}
\begin{tabular}{r*{11}r@{}r}
\toprule
 & \multicolumn{12}{c}{Degree}\\
\cmidrule{2-13}
$n$ & 1 & 2 & 5 & 10 & 20 & 30 & 40 & 50 & 100 & 200 & 300 & \\
\midrule 2 & \minavgmax{      0}{      0}{      0}
 & \minavgmax{      0}{      0}{    0.1}
 & \minavgmax{      0}{      0}{    0.1}
 & \minavgmax{      0}{      0}{    0.1}
 & \minavgmax{      0}{      0}{    0.1}
 & \minavgmax{      0}{      0}{    0.1}
 & \minavgmax{      0}{    0.1}{    0.2}
 & \minavgmax{      0}{    0.1}{    0.3}
 & \minavgmax{      0}{    0.5}{    1.2}
 & \minavgmax{    0.3}{    3.0}{    7.1}
 & \minavgmax{    0.2}{    8.1}{     30}
 & 
\\[7ex]
3 & \minavgmax{      0}{      0}{      0}
 & \minavgmax{      0}{      0}{      0}
 & \minavgmax{      0}{      0}{    0.1}
 & \minavgmax{      0}{    0.1}{    0.2}
 & \minavgmax{      0}{    0.4}{    6.2}
 & \minavgmax{      0}{    0.5}{    3.7}
 & \minavgmax{    0.1}{    1.2}{    7.1}
 & \minavgmax{    0.1}{    3.4}{     11}
 & \minavgmax{    0.2}{     34}{    164}
 & 
\\[7ex]
4 & \minavgmax{      0}{      0}{      0}
 & \minavgmax{      0}{      0}{    0.1}
 & \minavgmax{      0}{    0.1}{    0.2}
 & \minavgmax{      0}{    0.3}{    0.7}
 & \minavgmax{    0.1}{    1.4}{     12}
 & \minavgmax{    0.3}{    4.8}{     35}
 & \minavgmax{    0.4}{     16}{     77}
 & \minavgmax{    1.0}{     39}{    176}
 & 
\\[7ex]
5 & \minavgmax{      0}{    0.1}{    0.1}
 & \minavgmax{      0}{    0.1}{    0.1}
 & \minavgmax{    0.1}{    0.2}{    0.4}
 & \minavgmax{    0.1}{    0.7}{     12}
 & \minavgmax{    0.1}{    4.5}{     35}
 & \minavgmax{    1.4}{     36}{    551}
 & \minavgmax{    0.2}{     78}{    353}
 & 
\\[7ex]
6 & \minavgmax{    0.1}{    0.1}{    0.1}
 & \minavgmax{      0}{    0.2}{    5.0}
 & \minavgmax{    0.1}{    0.3}{    0.6}
 & \minavgmax{    0.2}{    1.6}{    5.9}
 & \minavgmax{    0.4}{     24}{    205}
 & 
\\[7ex]
7 & \minavgmax{    0.1}{    0.1}{    0.2}
 & \minavgmax{    0.1}{    0.3}{    5.4}
 & \minavgmax{    0.2}{    0.7}{    4.6}
 & \minavgmax{    0.4}{    4.1}{     22}
 & 
\\[7ex]
8 & \minavgmax{    0.2}{    0.2}{    0.2}
 & \minavgmax{    0.2}{    0.3}{    4.8}
 & \minavgmax{    0.2}{    1.0}{    4.1}
 & \minavgmax{    0.3}{     11}{    111}
 & 
\\[7ex]
10 & \minavgmax{    0.3}{    0.5}{    7.0}
 & \minavgmax{    0.4}{    0.5}{    3.5}
 & \minavgmax{    0.4}{    2.8}{     12}
 & 
\\[7ex]
15 & \minavgmax{    1.3}{    1.7}{    5.8}
 & \minavgmax{    1.4}{    1.9}{    5.1}
 & \minavgmax{    3.5}{     25}{     73}
 & 
\\[7ex]
20 & \minavgmax{    3.8}{    5.0}{    9.4}
 & \minavgmax{    4.0}{    5.3}{    8.1}
 & \minavgmax{    4.7}{    123}{    352}
 & 
\\[7ex]
30 & \minavgmax{     25}{     29}{     41}
 & \minavgmax{     26}{     29}{     32}
 & 
\\[7ex]
40 & \minavgmax{     88}{     98}{    106}
 & \minavgmax{     90}{    101}{    152}
 & 
\\[7ex]
50 & \minavgmax{    248}{    271}{    300}
 & \minavgmax{    259}{    283}{    429}
 & 
\\
\bottomrule
\end{tabular}
  \end{center}
\end{table}

\clearpage

\subsection{Integration of a monomial, using  Taylor expansion.}

Here,  we test the implementation  of   the algorithm described in
\autoref{s:fixed_degree}. This algorithm is based on
\autoref{th:expansion} .\footnote{This method is
  implemented in the Maple procedure
  \emph{integral\_via\_duality}, defined in the file \texttt{duality.mpl}.}

The running times are shown in
\autoref{tab:integration-via-duality}.

\begin{table}[ht]
  \caption{Integration of a random monomial of prescribed degree  using Taylor expansion.}
  \label{tab:integration-via-duality}
  \tiny
  \begin{center}
\begin{tabular}{r*{11}r@{}r}
\toprule
 & \multicolumn{12}{c}{Degree}\\
\cmidrule{2-13}
$n$ & 1 & 2 & 5 & 10 & 20 & 30 & 40 & 50 & 100 & 200 & 300 & \\
\midrule
2 & \minavgmax{      0}{      0}{      0}
 & \minavgmax{      0}{      0}{      0}
 & \minavgmax{      0}{      0}{      0}
 & \minavgmax{      0}{      0}{    0.1}
 & \minavgmax{      0}{      0}{    0.1}
 & \minavgmax{      0}{    0.1}{    0.3}
 & \minavgmax{      0}{    0.2}{    0.6}
 & \minavgmax{      0}{    0.3}{    0.9}
 & \minavgmax{      0}{    2.2}{     13}
 & \minavgmax{    0.9}{     24}{    101}
 & \minavgmax{    1.9}{     68}{    426}



\\[7ex]
3 & \minavgmax{      0}{      0}{      0}
 & \minavgmax{      0}{      0}{      0}
 & \minavgmax{      0}{      0}{    0.1}
 & \minavgmax{      0}{    0.1}{    0.4}
 & \minavgmax{      0}{    2.0}{     12}
 & \minavgmax{      0}{    8.1}{     60}
 & \minavgmax{    0.1}{     39}{    277}
 & \minavgmax{      0}{     61}{    512}
 & 
\\[7ex]
4 & \minavgmax{      0}{      0}{      0}
 & \minavgmax{      0}{      0}{      0}
 & \minavgmax{      0}{    0.1}{    0.4}
 & \minavgmax{      0}{    1.5}{     34}
 & 
\\[7ex]
5 & \minavgmax{      0}{      0}{      0}
 & \minavgmax{      0}{      0}{      0}
 & \minavgmax{      0}{    0.1}{    0.6}
 & \minavgmax{    0.1}{    4.9}{     48}
 & 
\\[7ex]
6 & \minavgmax{      0}{      0}{      0}
 & \minavgmax{      0}{      0}{      0}
 & \minavgmax{      0}{    0.4}{    1.7}
 & \minavgmax{    0.1}{     29}{    236}
 & 
\\[7ex]
7 & \minavgmax{      0}{      0}{      0}
 & \minavgmax{      0}{      0}{      0}
 & \minavgmax{      0}{    0.5}{    1.7}
 & 
\\[7ex]
8 & \minavgmax{      0}{      0}{      0}
 & \minavgmax{      0}{      0}{      0}
 & \minavgmax{      0}{    1.1}{     16}
 & 
\\[7ex]
10 & \minavgmax{      0}{      0}{      0}
 & \minavgmax{      0}{      0}{    0.1}
 & \minavgmax{      0}{    3.0}{     33}
 & 
\\[7ex]
15 & \minavgmax{      0}{      0}{      0}
 & \minavgmax{      0}{    0.1}{    0.1}
 & \minavgmax{    0.1}{     20}{     64}
 & 
\\[7ex]
20 & \minavgmax{      0}{      0}{      0}
 & \minavgmax{      0}{    0.7}{     28}
 & \minavgmax{    1.2}{     81}{    205}
 & 
\\[7ex]
30 & \minavgmax{      0}{      0}{    0.1}
 & \minavgmax{      0}{    0.7}{     15}
 & 
\\[7ex]
40 & \minavgmax{      0}{    0.1}{    0.1}
 & \minavgmax{      0}{    1.2}{     23}
 & 
\\[7ex]
50 & \minavgmax{    0.1}{    0.1}{    0.2}
 & \minavgmax{    0.1}{    1.7}{     18}
 & 
\\
\bottomrule
\end{tabular}
  \end{center}
\end{table}

\clearpage
\subsection{Integration of a dense homogeneous polynomial over a simplex}

Following the tests on single monomials, we ran tests on random
polynomials of varying density.  We generated these polynomials
using the Maple function~$\texttt{randpoly}$, requesting a
number~$r$ of monomials and the homogeneous degree~$M$. For each
monomial, the exponent vector was drawn uniformly from $\{\, \ve
M\in\N^d : |\ve M|=M\,\}$, and the coefficient is drawn uniformly
from $\{1,\dots,100\}.$ Due to collisions, the generated polynomial
actually can have fewer monomials than~$r$.\footnote{This is
implemented in
  the Maple procedure
  \texttt{random\_sparse\_\allowbreak{}homogeneous\_\allowbreak{}polynomial\_with\_degree}
  in the file \texttt{examples.mpl}.}

We only include the results for a family of randomly generated, very
dense homogeneous polynomials, where we draw $r =
\binom{M+d-1}{d-1}$ random monomials.
Table~\ref{tab:number-of-monomials-of-dense-poly} shows the number
of monomials in the resulting polynomials for our random tests.
Tables~\ref{tab:integration-dense-poly-by-decomposition-powerlinform},
\ref{tab:integration-dense-poly-via-duality},
\ref{tab:integration-dense-poly-via-iterated} show the test results
of the three methods.

We remark that in the case of the method using decompositions into powers of linear forms, we
note that the same powers of linear forms appear in the decomposition
formulas~\eqref{eq:decomp-powerlinform} for many different monomials $\ve
x^{\ve M_1}$, $\ve x^{\ve M_2}$.
We take advantage of this fact by collecting the coefficients of powers of
linear forms.\footnote{This is implemented in the procedure \texttt{list\_integral\_via\_waring}.}

\begin{table}[t]
  \caption{Number of monomials in a test family of random dense homogeneous
    polynomials of prescribed degree}
  \label{tab:number-of-monomials-of-dense-poly}
  \tiny
  \begin{center}
    \hspace*{-8em}
\begin{tabular}{r*{12}r}
\toprule
 & \multicolumn{12}{c}{Degree}\\
\cmidrule{2-13}
$n$ & 1 & 2 & 5 & 10 & 20 & 30 & 40 & 50 & 100 & 200 & 300 & 1000\\
\midrule
1 & \minavgmax{    1}{    1}{    1}
 & \minavgmax{    1}{    1}{    1}
 & \minavgmax{    1}{    1}{    1}
 & \minavgmax{    1}{    1}{    1}
 & \minavgmax{    1}{    1}{    1}
 & \minavgmax{    1}{    1}{    1}
 & \minavgmax{    1}{    1}{    1}
 & \minavgmax{    1}{    1}{    1}
 & \minavgmax{    1}{    1}{    1}
 & \minavgmax{    1}{    1}{    1}
 & \minavgmax{    1}{    1}{    1}
 & \minavgmax{    1}{    1}{    1}
\\[7ex]
2 & \minavgmax{    2}{    2}{    2}
 & \minavgmax{    1}{    2.2}{    3}
 & \minavgmax{    2}{    4.2}{    6}
 & \minavgmax{    4}{    7.1}{    9}
 & \minavgmax{     11}{     14}{     18}
 & \minavgmax{     16}{     20}{     25}
 & \minavgmax{     22}{     26}{     33}
 & \minavgmax{     28}{     33}{     37}
 & \minavgmax{     59}{     64}{     70}
 & \minavgmax{    120}{    128}{    140}
 & \minavgmax{    179}{    189}{    198}
 & \minavgmax{    606}{    634}{    659}
\\[7ex]
3 & \minavgmax{    3}{    3}{    3}
 & \minavgmax{    2}{    4}{    6}
 & \minavgmax{     11}{     13}{     17}
 & \minavgmax{     37}{     42}{     46}
 & \minavgmax{    133}{    146}{    160}
 & \minavgmax{    299}{    313}{    325}
 & \minavgmax{    524}{    543}{    560}
 & \minavgmax{    813}{    839}{    867}
 & \minavgmax{   3209}{   3262}{   3307}
 & \minavgmax{  12755}{  12836}{  12923}
 & 
\\[7ex]
4 & \minavgmax{    4}{    4}{    4}
 & \minavgmax{    4}{    6.5}{    8}
 & \minavgmax{     31}{     35}{     40}
 & \minavgmax{    168}{    182}{    196}
 & \minavgmax{   1085}{   1118}{   1144}
 & \minavgmax{   3396}{   3452}{   3499}
 & \minavgmax{   7748}{   7802}{   7887}
 & \minavgmax{  14719}{  14807}{  14914}
 & 
\\[7ex]
5 & \minavgmax{    5}{    5}{    5}
 & \minavgmax{    8}{    9.7}{     12}
 & \minavgmax{     72}{     81}{     88}
 & \minavgmax{    610}{    633}{    661}
 & \minavgmax{   6650}{   6716}{   6789}
 & \minavgmax{  29181}{  29329}{  29536}
 & 
\\[7ex]
6 & \minavgmax{    6}{    6}{    6}
 & \minavgmax{     10}{     14}{     16}
 & \minavgmax{    146}{    160}{    174}
 & \minavgmax{   1847}{   1896}{   1934}
 & \minavgmax{  33406}{  33591}{  33880}
 & 
\\[7ex]
7 & \minavgmax{    7}{    7}{    7}
 & \minavgmax{     15}{     18}{     21}
 & \minavgmax{    276}{    291}{    304}
 & \minavgmax{   5015}{   5062}{   5114}
 & 
\\[7ex]
8 & \minavgmax{    8}{    8}{    8}
 & \minavgmax{     19}{     23}{     29}
 & \minavgmax{    476}{    498}{    521}
 & \minavgmax{  12213}{  12295}{  12396}
 & 
\\[7ex]
10 & \minavgmax{     10}{     10}{     10}
 & \minavgmax{     29}{     36}{     41}
 & \minavgmax{   1236}{   1262}{   1289}
 & 
\\[7ex]
15 & \minavgmax{     15}{     15}{     15}
 & \minavgmax{     64}{     76}{     85}
 & \minavgmax{   7274}{   7352}{   7442}
 & 
\\[7ex]
20 & \minavgmax{     20}{     20}{     20}
 & \minavgmax{    123}{    133}{    146}
 & \minavgmax{  26742}{  26880}{  26985}
 & 
\\[7ex]
30 & \minavgmax{     30}{     30}{     30}
 & \minavgmax{    282}{    295}{    309}
 & 
\\[7ex]
40 & \minavgmax{     40}{     40}{     40}
 & \minavgmax{    502}{    521}{    542}
 & 
\\
\bottomrule
\end{tabular}
\end{center}
\end{table}

\begin{table}[t]
  \caption{Integration of a random dense homogeneous polynomial of prescribed degree by
    decomposition into a sum of powers of linear forms}
  \label{tab:integration-dense-poly-by-decomposition-powerlinform}
  \tiny
  \begin{center}
\begin{tabular}{r*{11}r@{}r}
\toprule
 & \multicolumn{12}{c}{Degree}\\
\cmidrule{2-13}
$n$ & 1 & 2 & 5 & 10 & 20 & 30 & 40 & 50 & 100 & 200 & 300 & \\
\midrule
2 & \minavgmax{      0}{      0}{      0}
 & \minavgmax{      0}{      0}{      0}
 & \minavgmax{      0}{      0}{    0.1}
 & \minavgmax{    0.1}{    0.1}{    0.2}
 & \minavgmax{    1.0}{    1.2}{    1.4}
 & \minavgmax{    3.1}{    3.6}{    4.4}
 & \minavgmax{    7.8}{    9.0}{     11}
 & \minavgmax{     16}{     18}{     21}
 & \minavgmax{    228}{    270}{    304}
 & 
\\[7ex]
3 & \minavgmax{      0}{      0}{    0.1}
 & \minavgmax{      0}{      0}{    0.1}
 & \minavgmax{    0.1}{    0.2}{    0.3}
 & \minavgmax{    2.1}{    2.5}{    2.9}
 & \minavgmax{     74}{     81}{     88}
 & 
\\[7ex]
4 & \minavgmax{      0}{      0}{    0.1}
 & \minavgmax{      0}{    0.1}{    0.1}
 & \minavgmax{    0.6}{    0.8}{    1.0}
 & \minavgmax{     30}{     32}{     35}
 & 
\\[7ex]
5 & \minavgmax{      0}{      0}{    0.2}
 & \minavgmax{    0.1}{    0.1}{    0.2}
 & \minavgmax{    2.6}{    3.0}{    3.4}
 & 
\\[7ex]
6 & \minavgmax{      0}{    0.1}{    0.2}
 & \minavgmax{    0.1}{    0.2}{    0.3}
 & \minavgmax{    8.9}{    9.5}{     10}
 & 
\\[7ex]
7 & \minavgmax{    0.1}{    0.1}{    0.2}
 & \minavgmax{    0.2}{    0.3}{    0.4}
 & \minavgmax{     26}{     28}{     30}
 & 
\\[7ex]
8 & \minavgmax{    0.1}{    0.1}{    0.2}
 & \minavgmax{    0.4}{    0.5}{    0.7}
 & \minavgmax{     81}{     85}{     91}
 & 
\\[7ex]
10 & \minavgmax{    0.2}{    0.3}{    0.4}
 & \minavgmax{    0.9}{    1.0}{    1.2}
 & 
\\[7ex]
15 & \minavgmax{    0.6}{    0.7}{    0.8}
 & \minavgmax{    4.0}{    4.3}{    4.8}
 & 
\\[7ex]
20 & \minavgmax{    1.6}{    1.8}{    1.9}
 & \minavgmax{     12}{     13}{     14}
 & 
\\[7ex]
30 & \minavgmax{    5.8}{    6.1}{    6.4}
 & \minavgmax{     63}{     68}{     72}
 & 
\\[7ex]
40 & \minavgmax{     15}{     16}{     16}
 & \minavgmax{    221}{    232}{    241}
 &
\\[7ex]
50 & \minavgmax{     29}{     30}{     30}
 & 
\\
\bottomrule
\end{tabular}
\end{center}
\end{table}
\begin{table}[t]
  \caption{Integration of a random dense homogeneous polynomial of prescribed
    degree using
    iterated Laurent series}
  \label{tab:integration-dense-poly-via-iterated}
  \tiny
  \begin{center}
\begin{tabular}{r*{11}r@{}r}
\toprule
 & \multicolumn{12}{c}{Degree}\\
\cmidrule{2-13}
$n$ & 1 & 2 & 5 & 10 & 20 & 30 & 40 & 50 & 100 & 200 & 300 & \\
\midrule 2 & \minavgmax{      0}{      0}{    0.1}
 & \minavgmax{      0}{      0}{    0.1}
 & \minavgmax{      0}{      0}{    0.1}
 & \minavgmax{      0}{    0.1}{    0.2}
 & \minavgmax{    0.2}{    0.3}{    0.4}
 & \minavgmax{    0.6}{    0.8}{    1.0}
 & \minavgmax{    1.2}{    1.7}{    2.2}
 & \minavgmax{    2.1}{    3.1}{    4.0}
 & \minavgmax{     24}{     29}{     37}
 & \minavgmax{    255}{    350}{    438}
 & 
\\[7ex]
3 & \minavgmax{      0}{      0}{    0.2}
 & \minavgmax{      0}{    0.1}{    0.2}
 & \minavgmax{    0.2}{    0.4}{    0.6}
 & \minavgmax{    2.1}{    2.7}{    3.3}
 & \minavgmax{     22}{     32}{     37}
 & \minavgmax{    146}{    184}{    210}
 & 
\\[7ex]
4 & \minavgmax{      0}{    0.1}{    0.2}
 & \minavgmax{    0.1}{    0.2}{    0.3}
 & \minavgmax{    2.4}{    3.0}{    3.6}
 & \minavgmax{     37}{     40}{     44}
 & 
\\[7ex]
5 & \minavgmax{    0.2}{    0.2}{    0.3}
 & \minavgmax{    0.3}{    0.5}{    0.7}
 & \minavgmax{     12}{     14}{     16}
 & \minavgmax{    351}{    365}{    380}
 & 
\\[7ex]
6 & \minavgmax{    0.3}{    0.4}{    0.5}
 & \minavgmax{    1.0}{    1.2}{    1.5}
 & \minavgmax{     50}{     55}{     61}
 & 
\\[7ex]
7 & \minavgmax{    0.7}{    0.8}{    0.9}
 & \minavgmax{    2.0}{    2.4}{    2.8}
 & \minavgmax{    161}{    179}{    195}
 & 
\\[7ex]
8 & \minavgmax{    1.2}{    1.3}{    1.6}
 & \minavgmax{    3.3}{    4.6}{    5.8}
 & \minavgmax{    481}{    517}{    560}
 & 
\\[7ex]
10 & \minavgmax{    3.5}{    3.6}{    3.8}
 & \minavgmax{     13}{     15}{     18}
 & 
\\[7ex]
15 & \minavgmax{     22}{     23}{     23}
 & \minavgmax{    117}{    127}{    140}
 & 
\\[7ex]
20 & \minavgmax{     91}{     94}{     98}
 & 
\\
\bottomrule
\end{tabular}
  \end{center}
\end{table}

\begin{table}[t]
  \caption{Integration of a random dense homogeneous polynomial of prescribed degree  using Taylor expansion. }
  \label{tab:integration-dense-poly-via-duality}
  \tiny
  \begin{center}
\begin{tabular}{r*{11}r@{}r}
\toprule
 & \multicolumn{12}{c}{Degree}\\
\cmidrule{2-13}
$n$ & 1 & 2 & 5 & 10 & 20 & 30 & 40 & 50 & 100 & 200 & 300 & \\
\midrule
2 & \minavgmax{      0}{      0}{      0}
 & \minavgmax{      0}{      0}{      0}
 & \minavgmax{      0}{      0}{    0.1}
 & \minavgmax{      0}{    0.1}{    0.1}
 & \minavgmax{    0.1}{    0.2}{    0.3}
 & \minavgmax{    0.3}{    0.8}{    1.1}
 & \minavgmax{    1.2}{    2.3}{    3.1}
 & \minavgmax{    1.2}{    5.6}{    7.4}
 & \minavgmax{     15}{     98}{    129}
 & 
\\[7ex]
3 & \minavgmax{      0}{      0}{      0}
 & \minavgmax{      0}{      0}{      0}
 & \minavgmax{    0.1}{    0.3}{    2.0}
 & \minavgmax{    1.1}{    2.4}{    3.3}
 & \minavgmax{     60}{    132}{    158}
 & 
\\[7ex]
4 & \minavgmax{      0}{      0}{      0}
 & \minavgmax{      0}{      0}{    0.1}
 & \minavgmax{    1.1}{    1.6}{    2.6}
 & \minavgmax{     62}{    117}{    143}
 & 
\\[7ex]
5 & \minavgmax{      0}{      0}{      0}
 & \minavgmax{    0.1}{    0.1}{    0.1}
 & \minavgmax{    4.4}{    8.1}{     24}
 & 
\\[7ex]
6 & \minavgmax{      0}{      0}{    0.1}
 & \minavgmax{    0.1}{    0.2}{    2.3}
 & \minavgmax{     28}{     35}{     41}
 & 
\\[7ex]
7 & \minavgmax{      0}{      0}{    0.1}
 & \minavgmax{    0.2}{    0.4}{    5.5}
 & \minavgmax{    117}{    136}{    150}
 & 
\\[7ex]
8 & \minavgmax{      0}{    0.1}{    0.1}
 & \minavgmax{    0.4}{    0.7}{     12}
 & \minavgmax{    366}{    431}{    475}
 & 
\\[7ex]
10 & \minavgmax{    0.1}{    0.1}{    0.1}
 & \minavgmax{    0.8}{    1.3}{    9.8}
 & 
\\[7ex]
15 & \minavgmax{    0.1}{    0.2}{    0.2}
 & \minavgmax{    4.0}{    5.6}{    7.3}
 & 
\\[7ex]
20 & \minavgmax{    0.3}{    0.4}{    2.8}
 & \minavgmax{     16}{     18}{     20}
 & 
\\[7ex]
30 & \minavgmax{    0.9}{    1.1}{    4.3}
 & \minavgmax{     95}{    104}{    109}
 & 
\\[7ex]
40 & \minavgmax{    2.3}{    2.6}{    4.4}
 & \minavgmax{    370}{    384}{    403}
 &
\\[7ex]
50 & \minavgmax{    5.2}{    5.3}{    5.5}
 & 
\\
\bottomrule
\end{tabular}
  \end{center}
\end{table}

\subsection{Integration of a monomial
  with few effective variables}

Finally, we tested the performance of the three algorithms on
monomials~$\ve x^{\ve M}$ with a small number~$D$ of effective
variables.  We fix the number~$D$.  Then, for a given
dimension~$n\geq D$ and total degree~$M$, we picked 50 combinations
of a random simplex~$\Delta$ of dimension~$n$ and a random exponent
vector~$\ve M = (M_1,\dots,M_D,0,\dots,0)$ with $|\ve M| = M$. We
only include the results for~$D=2$ in Tables
\ref{tab:integration-few-effective-by-decomposition-powerlinform},
\ref{tab:integration-few-effective-via-duality}, and
\ref{tab:integration-few-effective-via-iterated}.

\begin{table}[t]
  \caption{Integration of a monomial of prescribed degree
    with 2 effective variables by
    decomposition into a sum of powers of linear forms}
  \label{tab:integration-few-effective-by-decomposition-powerlinform}
  \tiny
  \begin{center}
\begin{tabular}{r*{11}r@{}r}
\toprule
 & \multicolumn{12}{c}{Degree}\\
\cmidrule{2-13}
$n$ & 1 & 2 & 5 & 10 & 20 & 30 & 40 & 50 & 100 & 200 & 300 & \\
\midrule
3 & \minavgmax{      0}{      0}{      0}
 & \minavgmax{      0}{      0}{    0.2}
 & \minavgmax{      0}{      0}{    0.1}
 & \minavgmax{      0}{    0.1}{    0.2}
 & \minavgmax{      0}{    0.3}{    0.5}
 & \minavgmax{      0}{    0.7}{    1.1}
 & \minavgmax{      0}{    1.5}{    2.3}
 & \minavgmax{      0}{    2.2}{    3.5}
 & \minavgmax{    1.9}{     11}{     15}
 & \minavgmax{    1.9}{     47}{     80}
 & \minavgmax{    9.8}{    109}{    217}
 & 
\\[7ex]
4 & \minavgmax{      0}{      0}{      0}
 & \minavgmax{      0}{      0}{    0.3}
 & \minavgmax{      0}{      0}{    0.2}
 & \minavgmax{      0}{    0.1}{    0.2}
 & \minavgmax{      0}{    0.4}{    0.7}
 & \minavgmax{      0}{    0.9}{    1.4}
 & \minavgmax{      0}{    1.7}{    2.8}
 & \minavgmax{    0.5}{    2.8}{    4.3}
 & \minavgmax{    0.1}{     13}{     20}
 & \minavgmax{    0.2}{     64}{    103}
 & \minavgmax{    0.3}{    180}{    270}
 & 
\\[7ex]
5 & \minavgmax{      0}{      0}{      0}
 & \minavgmax{      0}{      0}{    0.4}
 & \minavgmax{      0}{    0.1}{    0.2}
 & \minavgmax{      0}{    0.1}{    0.3}
 & \minavgmax{      0}{    0.6}{    1.0}
 & \minavgmax{      0}{    1.3}{    2.0}
 & \minavgmax{      0}{    2.5}{    3.8}
 & \minavgmax{      0}{    3.5}{    5.9}
 & \minavgmax{    1.8}{     17}{     26}
 & \minavgmax{    0.3}{     81}{    128}
 & \minavgmax{    6.0}{    205}{    335}
 & 
\\[7ex]
6 & \minavgmax{      0}{      0}{    0.2}
 & \minavgmax{      0}{      0}{    0.2}
 & \minavgmax{      0}{    0.1}{    0.2}
 & \minavgmax{      0}{    0.2}{    0.4}
 & \minavgmax{      0}{    0.7}{    1.2}
 & \minavgmax{      0}{    1.8}{    2.6}
 & \minavgmax{      0}{    3.1}{    4.9}
 & \minavgmax{      0}{    4.5}{    7.5}
 & \minavgmax{    2.1}{     25}{     34}
 & \minavgmax{    7.0}{    105}{    164}
 & \minavgmax{    0.3}{    241}{    449}
 & 
\\[7ex]
7 & \minavgmax{      0}{      0}{      0}
 & \minavgmax{      0}{      0}{    0.3}
 & \minavgmax{      0}{    0.1}{    0.3}
 & \minavgmax{      0}{    0.2}{    0.5}
 & \minavgmax{      0}{    1.0}{    1.5}
 & \minavgmax{      0}{    1.9}{    3.0}
 & \minavgmax{      0}{    3.7}{    5.9}
 & \minavgmax{      0}{    5.9}{    8.7}
 & \minavgmax{    2.3}{     24}{     39}
 & \minavgmax{    0.2}{    128}{    184}
 & \minavgmax{    0.2}{    279}{    468}
 & 
\\[7ex]
8 & \minavgmax{      0}{      0}{      0}
 & \minavgmax{      0}{      0}{    0.5}
 & \minavgmax{      0}{    0.1}{    0.3}
 & \minavgmax{      0}{    0.3}{    0.6}
 & \minavgmax{      0}{    1.1}{    1.6}
 & \minavgmax{      0}{    2.3}{    3.5}
 & \minavgmax{      0}{    4.0}{    6.5}
 & \minavgmax{    0.1}{    6.8}{   10.0}
 & \minavgmax{    4.2}{     28}{     45}
 & \minavgmax{    0.4}{    122}{    206}
 & \minavgmax{   10.0}{    347}{    529}
 & 
\\[7ex]
10 & \minavgmax{      0}{      0}{    0.3}
 & \minavgmax{      0}{      0}{    0.3}
 & \minavgmax{      0}{    0.1}{    0.4}
 & \minavgmax{      0}{    0.4}{    0.7}
 & \minavgmax{      0}{    1.3}{    2.2}
 & \minavgmax{      0}{    3.3}{    4.9}
 & \minavgmax{      0}{    5.6}{    8.6}
 & \minavgmax{    1.7}{     10}{     14}
 & \minavgmax{    0.1}{     41}{     58}
 & \minavgmax{     13}{    189}{    266}
 & 
\\[7ex]
15 & \minavgmax{      0}{    0.1}{    0.3}
 & \minavgmax{      0}{    0.1}{    0.3}
 & \minavgmax{      0}{    0.2}{    0.5}
 & \minavgmax{      0}{    0.7}{    1.2}
 & \minavgmax{    0.1}{    2.4}{    4.0}
 & \minavgmax{    0.1}{    6.0}{    8.9}
 & \minavgmax{    0.1}{    9.9}{     16}
 & \minavgmax{    0.1}{     15}{     24}
 & \minavgmax{    0.2}{     65}{    105}
 & \minavgmax{    0.2}{    292}{    479}
 & 
\\[7ex]
20 & \minavgmax{      0}{    0.1}{    0.3}
 & \minavgmax{    0.1}{    0.2}{    0.5}
 & \minavgmax{      0}{    0.5}{    0.8}
 & \minavgmax{    0.1}{    1.3}{    2.1}
 & \minavgmax{    0.1}{    4.7}{    7.1}
 & \minavgmax{    0.1}{    9.7}{     15}
 & \minavgmax{    0.1}{     17}{     26}
 & \minavgmax{    0.2}{     26}{     40}
 & \minavgmax{     23}{    123}{    170}
 & 
\\[7ex]
30 & \minavgmax{    0.1}{    0.2}{    0.4}
 & \minavgmax{    0.2}{    0.4}{    0.8}
 & \minavgmax{    0.2}{    0.9}{    1.5}
 & \minavgmax{    0.2}{    2.8}{    4.7}
 & \minavgmax{    0.2}{     11}{     15}
 & \minavgmax{    0.2}{     19}{     33}
 & \minavgmax{    0.2}{     38}{     58}
 & \minavgmax{    0.2}{     55}{     92}
 & \minavgmax{    0.3}{    254}{    369}
 & 
\\[7ex]
40 & \minavgmax{    0.3}{    0.4}{    0.6}
 & \minavgmax{    0.3}{    0.6}{    1.2}
 & \minavgmax{    0.3}{    1.9}{    2.9}
 & \minavgmax{    0.3}{    5.4}{    8.5}
 & \minavgmax{    0.3}{     20}{     29}
 & \minavgmax{    0.4}{     41}{     62}
 & \minavgmax{    0.4}{     66}{    110}
 & \minavgmax{    0.4}{    101}{    171}
 & 
\\[7ex]
50 & \minavgmax{    0.5}{    0.6}{    0.8}
 & \minavgmax{    0.5}{    1.0}{    2.1}
 & \minavgmax{    0.5}{    3.4}{    4.7}
 & \minavgmax{    0.6}{    9.2}{     14}
 & \minavgmax{    0.6}{     31}{     49}
 & \minavgmax{    0.6}{     63}{    106}
 & \minavgmax{    0.7}{    130}{    185}
 & \minavgmax{     36}{    201}{    286}
 & 
\\
\bottomrule
\end{tabular}
\end{center}
\end{table}

\begin{table}[t]
  \caption{
    Integration of a monomial of prescribed degree
    with 2 effective variables
    using iterated Laurent expansion}
  \label{tab:integration-few-effective-via-iterated}
  \tiny
  \begin{center}
\begin{tabular}{r*{11}r@{}r}
\toprule
 & \multicolumn{12}{c}{Degree}\\
\cmidrule{2-13}
$n$ & 1 & 2 & 5 & 10 & 20 & 30 & 40 & 50 & 100 & 200 & 300 & \\
\midrule
3 & \minavgmax{      0}{      0}{      0}
 & \minavgmax{      0}{      0}{      0}
 & \minavgmax{      0}{      0}{      0}
 & \minavgmax{      0}{      0}{    0.1}
 & \minavgmax{      0}{    0.1}{    0.1}
 & \minavgmax{      0}{    0.2}{    0.3}
 & \minavgmax{      0}{    0.3}{    5.6}
 & \minavgmax{      0}{    0.3}{    2.1}
 & \minavgmax{    0.2}{    1.2}{    3.6}
 & \minavgmax{    0.3}{    6.2}{     20}
 & \minavgmax{    0.9}{     16}{     62}
 & 
\\[7ex]
4 & \minavgmax{      0}{      0}{    0.1}
 & \minavgmax{      0}{      0}{    0.1}
 & \minavgmax{      0}{    0.1}{    0.1}
 & \minavgmax{      0}{    0.1}{    0.1}
 & \minavgmax{    0.1}{    0.1}{    0.2}
 & \minavgmax{    0.1}{    0.3}{    4.2}
 & \minavgmax{    0.1}{    0.3}{    2.5}
 & \minavgmax{    0.1}{    0.5}{    3.0}
 & \minavgmax{    0.2}{    2.1}{    5.4}
 & \minavgmax{    0.4}{     11}{     38}
 & \minavgmax{    0.9}{     47}{    166}
 & 
\\[7ex]
5 & \minavgmax{    0.1}{    0.1}{    0.1}
 & \minavgmax{      0}{    0.1}{    0.1}
 & \minavgmax{    0.1}{    0.1}{    0.1}
 & \minavgmax{    0.1}{    0.1}{    0.2}
 & \minavgmax{    0.1}{    0.3}{    6.4}
 & \minavgmax{    0.1}{    0.3}{    2.3}
 & \minavgmax{    0.1}{    0.5}{    2.8}
 & \minavgmax{    0.1}{    0.7}{    3.7}
 & \minavgmax{    0.3}{    3.6}{    9.8}
 & \minavgmax{    0.6}{     21}{     80}
 & \minavgmax{    1.3}{     86}{    271}
 & 
\\[7ex]
6 & \minavgmax{    0.1}{    0.1}{    0.2}
 & \minavgmax{    0.1}{    0.1}{    0.1}
 & \minavgmax{    0.1}{    0.1}{    0.2}
 & \minavgmax{    0.1}{    0.2}{    3.4}
 & \minavgmax{    0.1}{    0.3}{    2.5}
 & \minavgmax{    0.1}{    0.6}{    2.6}
 & \minavgmax{    0.2}{    0.8}{    3.8}
 & \minavgmax{    0.2}{    1.1}{    4.7}
 & \minavgmax{    0.4}{    7.2}{     30}
 & \minavgmax{    1.0}{     39}{    269}
 & \minavgmax{    1.9}{    100}{    550}
 & 
\\[7ex]
7 & \minavgmax{    0.1}{    0.1}{    0.2}
 & \minavgmax{    0.1}{    0.1}{    0.2}
 & \minavgmax{    0.1}{    0.2}{    3.9}
 & \minavgmax{    0.1}{    0.3}{    2.5}
 & \minavgmax{    0.2}{    0.6}{    2.8}
 & \minavgmax{    0.2}{    0.8}{    3.4}
 & \minavgmax{    0.2}{    1.2}{    4.2}
 & \minavgmax{    0.3}{    2.0}{    6.7}
 & \minavgmax{    0.6}{    6.8}{     26}
 & \minavgmax{    1.2}{     94}{    260}
 & 
\\[7ex]
8 & \minavgmax{    0.2}{    0.2}{    0.3}
 & \minavgmax{    0.2}{    0.3}{    4.5}
 & \minavgmax{    0.2}{    0.3}{    2.3}
 & \minavgmax{    0.2}{    0.4}{    2.6}
 & \minavgmax{    0.2}{    0.7}{    3.1}
 & \minavgmax{    0.2}{    1.2}{    3.6}
 & \minavgmax{    0.3}{    1.9}{    5.8}
 & \minavgmax{    0.3}{    2.9}{     12}
 & \minavgmax{    0.7}{     16}{     42}
 & \minavgmax{    2.0}{     99}{    322}
 & 
\\[7ex]
10 & \minavgmax{    0.3}{    0.5}{    4.2}
 & \minavgmax{    0.3}{    0.4}{    2.3}
 & \minavgmax{    0.3}{    0.5}{    2.5}
 & \minavgmax{    0.4}{    0.7}{    3.0}
 & \minavgmax{    0.4}{    1.2}{    3.9}
 & \minavgmax{    0.5}{    2.4}{    6.5}
 & \minavgmax{    0.6}{    3.8}{     11}
 & \minavgmax{    0.6}{    5.9}{     17}
 & \minavgmax{    1.1}{     23}{     77}
 & 
\\[7ex]
15 & \minavgmax{    1.3}{    1.4}{    1.6}
 & \minavgmax{    1.3}{    1.5}{    1.9}
 & \minavgmax{    1.4}{    1.7}{    2.0}
 & \minavgmax{    1.4}{    2.1}{    3.4}
 & \minavgmax{    1.6}{    4.0}{    9.1}
 & \minavgmax{    1.7}{    7.3}{     21}
 & \minavgmax{    1.9}{     11}{     31}
 & \minavgmax{    2.4}{     20}{     52}
 & 
\\[7ex]
20 & \minavgmax{    4.0}{    4.2}{    4.4}
 & \minavgmax{    4.0}{    4.4}{    5.6}
 & \minavgmax{    4.1}{    4.8}{    5.7}
 & \minavgmax{    4.2}{    6.1}{    9.3}
\\[7ex]
30 & \minavgmax{     22}{     25}{     26}
 & \minavgmax{     24}{     25}{     28}
 & \minavgmax{     24}{     27}{     31}
 & \minavgmax{     25}{     32}{     45}
 & \minavgmax{     26}{     47}{     96}
 & 
\\[7ex]
40 & \minavgmax{     71}{     85}{     92}
 & \minavgmax{     88}{     92}{    102}
 & \minavgmax{     90}{     95}{    103}
\\[7ex]
50 & \minavgmax{    196}{    231}{    252}
 & \minavgmax{    228}{    239}{    249}
 & \minavgmax{    212}{    244}{    277}
 & \minavgmax{    217}{    261}{    347}
 & 
\\
\bottomrule
\end{tabular}
  \end{center}
\end{table}

\begin{table}[t]
  \caption{Integration of a monomial of prescribed degree
    with 2 effective variables
    using Taylor expansion}
  \label{tab:integration-few-effective-via-duality}
  \tiny
  \begin{center}
\noindent
\begin{tabular}{r*{11}r@{}r}
\toprule
 & \multicolumn{12}{c}{Degree}\\
\cmidrule{2-13}
$n$ & 1 & 2 & 5 & 10 & 20 & 30 & 40 & 50 & 100 & 200 & 300 & \\
\midrule
3 & \minavgmax{      0}{      0}{      0}
 & \minavgmax{      0}{      0}{      0}
 & \minavgmax{      0}{      0}{    0.1}
 & \minavgmax{      0}{      0}{    0.1}
 & \minavgmax{      0}{    0.1}{    0.2}
 & \minavgmax{      0}{    0.2}{    0.6}
 & \minavgmax{      0}{    0.4}{    1.3}
 & \minavgmax{      0}{    0.8}{    2.6}
 & \minavgmax{      0}{    5.8}{     27}
 & \minavgmax{    1.9}{     62}{    359}
\\[7ex]
4 & \minavgmax{      0}{      0}{      0}
 & \minavgmax{      0}{      0}{      0}
 & \minavgmax{      0}{      0}{    0.1}
 & \minavgmax{      0}{      0}{    0.1}
 & \minavgmax{      0}{    0.2}{    0.4}
 & \minavgmax{      0}{    0.4}{    1.2}
 & \minavgmax{      0}{    0.7}{    3.0}
 & \minavgmax{      0}{    1.8}{    8.2}
 & \minavgmax{      0}{     11}{     58}
\\[7ex]
5 & \minavgmax{      0}{      0}{      0}
 & \minavgmax{      0}{      0}{      0}
 & \minavgmax{      0}{      0}{    0.1}
 & \minavgmax{      0}{    0.1}{    0.2}
 & \minavgmax{      0}{    0.2}{    0.7}
 & \minavgmax{      0}{    0.6}{    2.2}
 & \minavgmax{      0}{    2.0}{    6.7}
 & \minavgmax{      0}{    2.9}{     15}
 & \minavgmax{      0}{     19}{    110}
 & 
\\[7ex]
6 & \minavgmax{      0}{      0}{      0}
 & \minavgmax{      0}{      0}{      0}
 & \minavgmax{      0}{      0}{    0.1}
 & \minavgmax{      0}{    0.1}{    0.3}
 & \minavgmax{      0}{    0.7}{     16}
 & \minavgmax{      0}{    0.9}{    8.4}
 & \minavgmax{      0}{    2.0}{    8.8}
 & \minavgmax{      0}{    4.4}{     72}
 & \minavgmax{      0}{     31}{    124}
\\[7ex]
7 & \minavgmax{      0}{      0}{      0}
 & \minavgmax{      0}{      0}{    0.1}
 & \minavgmax{      0}{      0}{    0.1}
 & \minavgmax{      0}{    0.1}{    0.3}
 & \minavgmax{      0}{    0.5}{    1.8}
 & \minavgmax{      0}{    1.6}{    8.4}
 & \minavgmax{    0.2}{    4.3}{     21}
 & \minavgmax{      0}{    5.9}{     32}
 & \minavgmax{      0}{     52}{    246}
\\[7ex]
8 & \minavgmax{      0}{      0}{    0.1}
 & \minavgmax{      0}{      0}{    0.1}
 & \minavgmax{      0}{    0.1}{    0.2}
 & \minavgmax{      0}{    0.2}{    0.4}
 & \minavgmax{      0}{    0.9}{    2.8}
 & \minavgmax{      0}{    2.7}{     36}
 & \minavgmax{      0}{    3.6}{     19}
 & \minavgmax{    0.3}{    9.3}{     42}
 & \minavgmax{    0.1}{     65}{    547}
\\[7ex]
10 & \minavgmax{      0}{      0}{    0.1}
 & \minavgmax{      0}{      0}{    0.1}
 & \minavgmax{      0}{    0.1}{    0.2}
 & \minavgmax{      0}{    0.2}{    0.6}
 & \minavgmax{      0}{    1.1}{    7.0}
 & \minavgmax{      0}{    3.5}{     22}
 & \minavgmax{      0}{    8.4}{     40}
 & \minavgmax{      0}{     20}{     92}
\\[7ex]
20 & \minavgmax{      0}{      0}{    0.1}
 & \minavgmax{      0}{    0.1}{    0.2}
 & \minavgmax{      0}{    0.4}{    1.1}
 & \minavgmax{      0}{    2.8}{    9.3}
 & \minavgmax{      0}{     12}{     43}
 & \minavgmax{      0}{     40}{    156}
 & \minavgmax{      0}{     74}{    309}
 & \minavgmax{    2.0}{    121}{    462}
\\[7ex]
30 & \minavgmax{      0}{      0}{    0.1}
 & \minavgmax{      0}{    0.2}{    0.5}
 & \minavgmax{      0}{    1.4}{    5.6}
 & \minavgmax{      0}{    6.9}{     29}
 & \minavgmax{      0}{     62}{    205}
 & \minavgmax{    0.1}{     90}{    583}
 & 
\\[7ex]
40 & \minavgmax{      0}{    0.1}{    0.1}
 & \minavgmax{      0}{    0.4}{    0.9}
 & \minavgmax{    0.1}{    3.3}{     32}
 & \minavgmax{    0.1}{     10}{     40}
 & \minavgmax{    0.1}{     71}{    329}
 & 
\\[7ex]
50 & \minavgmax{    0.1}{    0.1}{    0.1}
 & \minavgmax{    0.1}{    0.4}{    1.5}
 & \minavgmax{    0.1}{    5.5}{     39}
 & \minavgmax{    0.1}{     19}{     65}
 & \minavgmax{    0.1}{    126}{    525}
\\
\bottomrule
\end{tabular}
  \end{center}
\end{table}

\clearpage
\subsection{Discussion}

In our implementation of the three methods and our experiments for
the case of random monomials, we observe that the method of iterated
Laurent expansion is faster than the two other methods if the
dimension~$n$ is very small (up to $n=5$).  Starting from
dimension~$n=6$, the method using decompositions into powers of
linear forms is faster than the other two methods.  The method using
Taylor expansion  is always inferior to the better of the two other
methods, for any combination of degree and dimension.

In the experiments with random dense polynomials, in our implementation we did not
see significant savings from collecting the coefficients of the same powers of
linear forms.  As a consequence, the ranking of the three methods is the same
as it is in the case of random monomials.

The experiments with random monomials with few effective variables
show that all three methods benefit from using few effective
variables.  The greatest effect is on the method using
decompositions into powers of linear forms, where, for example, the
restriction to 2~effective variables allows to handle combinations
of high degree~$M=200$ and high dimension $n=15$. However, for low
dimensions ($n\leq 5$), the method of iterated Laurent expansion
still wins.  Also here the method using Taylor expansion is always
inferior to the better of the two other methods. This discussion
shows the power of Brion's formula.

\clearpage
\section{Conclusions}

We discussed various algorithms for the exact integration of
polynomials over simplicial regions. Beside their theoretical
efficiency, the simple rough experiments we performed clearly
demonstrated that these methods are robust enough to attack rather
difficult problems.  Our investigations opened several doors for
further development, which we will present in a forthcoming paper.

First, we have some theoretical issues expanding on our results. As
in the case of volumes and the computation of centroids, it is
likely that our hardness result, Theorem \ref{generalMhard}, can be
extended into an  inapproximability result as those obtained in
\cite{rademacher}. Another  goal is to study other families of
polytopes  for which exact integration can be done efficiently.
Furthermore,  we will  present a natural extension of the
computation of integrals, the efficient computation of  the highest
degree coefficients of a weighted Ehrhart quasipolynomial of a
simplex. Besides the methods of the present article, these last
computations are based on the results of
\cite{Baldoni-Berline-Vergne-2008} and \cite{Berline-Vergne-2007}.

Second, our intention has been all along to develop algorithms with
a good chance of becoming practical and that allow for clear
implementation.  Thus we have also some practical improvements to
discuss.

Finally, in order to
develop practical integration software, it appears that our methods
should be coupled with fast techniques for decomposing domains into
polyhedral regions (e.g. triangulations).

\section{Acknowledgements}

A part of the work for this paper was done while the five authors
visited Centro di Ricerca Matematica Ennio De Giorgi at the Scuola
Normale Superiore of Pisa, Italy.  We are grateful for the
hospitality of CRM and the support we received. The third author was
also supported by NSF grant DMS-0608785. We are grateful for help
from Christophe Margerin and comments from Jean-Bernard Lasserre and
Bernd Sturmfels.

\clearpage
\bibliographystyle{amsabbrv}
\bibliography{biblio}

\end{document}